\documentclass[a4paper,10pt,twoside]{article}
\usepackage{amsmath,amsthm,amsfonts}
\usepackage[pdfborder={0 0 0},pagebackref=false]{hyperref}
\usepackage{geometry}
\title{Extremes and gaps in sampling from a GEM random discrete distribution}
\author{Jim Pitman\footnote{Statistics Department, 367 Evans Hall \# 3860, University of California, Berkeley, CA 94720-3860, U.S.A.
    \texttt{pitman@berkeley.edu}}
  \and %% remove this line and below if single author
 Yuri Yakubovich\footnote{Saint Petersburg State University, St.\;Petersburg State University, 7/9 Universitetskaya nab., St.\;Petersburg, 199034 Russia. \texttt{y.yakubovich@spbu.ru}}}

\newcommand{\UU}[2]{U^{[#2]}_{#1}}
\newcommand{\XX}[2]{X^{[#2]}_{#1}}

\newcommand{\GEM}{\mathsf{GEM}}
\newcommand{\RAM}{\mathsf{RAM}}
\newcommand{\GRAM}{\mathsf{GRAM}}

\newcommand{\ISBP}{\mathsf{ISBP}}

\newcommand{\ptau}{p}

\newcommand{\Pdec}{P^{\downarrow}}
\newcommand{\BN}{\mathbb{N}}

\newcommand{\BP}{\mathbb{P}}
\newcommand{\YY}{F}
\renewcommand{\P}{\mathbb{P}}
\newcommand{\BE}{\mathbb{E}}

\newcommand{\Var}{\operatorname{Var}}
\newcommand{\re}{\mathrm{e}}
\newcommand{\ii}{\mathrm{i}}
\newcommand{\PD}{\mathrm{PD}}

\newcommand{\deq}{\overset{d}{{}={}}}
\newcommand{\eps}{\varepsilon}

\newcommand{\Pbul}{P_{\bullet}}

\newcommand{\ed}{\overset{d}{{}={}}}

\newcommand{\convd}{\overset{d}{{}\rightarrow{}}}

\newcommand{\ind}{\mathbf{1}}

\newcommand{\giv}{\,|\,}  %% given,  for conditioning

\newcommand{\Nright}{N^{ X \uparrow} }
\newcommand{\Ndec}{N^{\downarrow}}

\newcommand{\Nstar}{N^{*}}

\newcommand{\Nbig}{m}

\renewcommand{\Re}{\operatorname{Re}}

\newtheorem{theorem}{Theorem}
\newtheorem{proposition}[theorem]{Proposition}
\newtheorem{lemma}[theorem]{Lemma}
\newtheorem{corollary}[theorem]{Corollary}
\theoremstyle{definition}
\newtheorem*{remark*}{Remark}
\newtheorem{remark}{Remark}

\numberwithin{equation}{section}
\numberwithin{theorem}{section}
\numberwithin{remark}{section}

\begin{document}

\maketitle

\begin{abstract}
We show that in a sample  of size $n$ from a $\GEM(0,\theta)$ random discrete distribution, the gaps $G_{i:n}:= X_{n-i+1:n} - X_{n-i:n}$ between order statistics $X_{1:n} \le \cdots \le X_{n:n}$ of 
the sample, with the convention $G_{n:n} := X_{1:n} - 1$, are distributed like the first $n$ terms of an infinite sequence of independent geometric$(i/(i+\theta))$ variables $G_i$.
This extends a known result for the minimum $X_{1:n}$ to other gaps in the range of the sample, and implies that the maximum $X_{n:n}$ 
has the distribution of $1 + \sum_{i=1}^n G_i$, hence the known result that $X_{n:n}$ grows like $\theta\log(n)$ as $n\to\infty$, with an asymptotically normal distribution.
Other consequences include most known formulas for the exact distributions of $\GEM(0,\theta)$ sampling statistics, including the Ewens and  Donnelly--Tavar\'e sampling formulas.
For the two-parameter GEM$(\alpha,\theta)$ distribution we show that the maximal value grows like a random multiple of $n^{\alpha/(1-\alpha)}$ and find the limit distribution of the multiplier.
\end{abstract}

\tableofcontents

\newpage

\section{Introduction}

Consider a sequence of real random variables
$X_1,X_2,\dots$ with order statistics
\[%begin{equation}
%\label{eq:defmax}
X_{1:n}:= \min_{1 \le i \le n} X_i \le X_{2:n} \le X_{3:n} \le  \cdots \le X_{n:n} := \max_{1 \le i \le n} X_i .
\]%end{equation}
For an independent and identically distributed (i.i.d.)\ sequence $(X_n)$ with a continuous distribution function $F(x):= \BP(X_n \le x )$,
the probabilistic structure of order statistics is well understood. Many exact distributional identities are obtained by reduction to the simplest i.i.d.\  cases:
\begin{itemize} 
\item $X_i = U_i$, signifying $F(x) = x$ for $0 \le x \le 1$, the {\em uniform distribution} on $(0,1)$; 
\item $X_i = \eps_i/\lambda $ for $\eps_i$ i.i.d.\  exponential$(1)$ and a constant $\lambda >0$, in which case $F(x) = 1 - e^{- \lambda x}$ for $x \ge 0$, the {\em exponential$(\lambda)$ distribution} on $(0,\infty)$. 
\end{itemize} 
This reduction involves the identity of $n$-dimensional joint distributions
\begin{equation}
\label{chvar}
(F( X_{i:n} ), 1 \le i \le n ) \ed ( U_{i:n}, 1 \le i \le n ) \ed (1 - \exp( - \eps_{i:n} ) , 1 \le i \le n)
\end{equation}
which holds with almost sure identities if the $U_i$ and $\eps_i$ are defined by $U_i:= F(X_i)$ and $\eps_i:= - \log(1 - U_i)$.
For discrete distributions the situation is complicated by possible ties but also well understood.
See \cite{MR1994955} %%%AUTHOR = {David, H. A. and Nagaraja, H. N.}, TITLE = {Order statistics},
\cite{MR1791071} %%%Author = Valery B. Nevzorov
for further background on order statistics. 

We are primarily interested here in the structure of the {\em gaps between sample values} which we list from the top of the sample down,
as
\[%begin{equation}
%\label{gaps}
G_{i:n} := X_{n+1-i:n} - X_{n-i:n}   \qquad (1 \le i \le n)
\]%end{equation}
for distributions of $X_i$ whose support has a {\em minimal value} $m_0 \ge 0$, with $X_{0:n}:= m_0$.
So we interpret $G_{n:n}:= X_{1:n} - m_0$ as the {\em gap below the minimum} of the sample, with $G_{n:n} = 0$ iff some sample value hits the minimum of the range.
The order statistics are then encoded in the gaps as $X_{k:n} = m_0 + \sum_{i = 0}^{k-1} G_{n-i:n}$.
In particular, the minimal and maximal values of the sample are 
\begin{equation}
\label{minmax}
X_{1:n}:= m_0 + G_{n:n} \qquad \mbox{ and } \qquad X_{n:n} = m_0 + \sum_{i = 1}^n G_{i:n}\,.
\end{equation}  
According to a well known result of Sukhatme--R{\'e}nyi 
\cite{MR0061792}, %%%, Alfr{\'e}d}, TITLE = {On the theory of order statistics},
\cite[Repr.~3.4]{MR1791071},
for i.i.d.\  sampling the structure of the gaps is simplest for exponential variables:
\begin{equation}
\label{explgaps}
\mbox{ for } X_i = \frac{ \eps_i } {\lambda} \mbox{ the $G_{i:n}$ are independent with } (G_{i:n} , 1 \le i \le n ) \ed \left( \frac{X_i}{i}, 1 \le i \le n \right).
\end{equation}
Among absolutely continuous distributions, the family of shifted exponential distributions is characterized by quite weak forms of this assertion, 
for instance that $G_{1:2}$ is independent of $G_{2:2}$. 
See Ferguson \cite{MR0226804} %%, AUTHOR = {Ferguson, Thomas S.}, TITLE = {On characterizing distributions by properties of order statistics},
 and earlier work cited there, and \cite{MR3474742} for more recent results in this vein. %%%AUTHOR = {Yanev, George P. and Chakraborty, Santanu}, TITLE = {A characterization of exponential distribution and the {S}ukhatme-{R}\'enyi decomposition of 
%The maximum of an exponential sample is just the sum of its gaps. So 
Formulas \eqref{minmax} and \eqref{explgaps} explain the well known identity in distribution
\begin{equation}
\label{expid}
M_n:= \max_{1 \le i \le n } \eps_i \ed T_n:= \sum_{i = 1}^n \frac{ \eps_i}{i}
\end{equation}
%relating the distribution of the maximum $M_n$ of $n$ i.i.d.\  exponential $(1)$ variables $\eps_i$ to that of the
%sum $T_n$ of $n$ independent exponential$(i)$ variables for $1 \le i \le n$.
%%This identity 
which implies the convergence in distribution, with centering but no normalization 
\begin{equation}
\label{explim}
M_n - \log n \,\, \ed \,\, T_n - \log n \,\, \convd  \,\, - \gamma + \sum_{i=1}^\infty \frac{ (\eps_i - 1 )}{i}
\end{equation}
where $\gamma:= \lim_{n \to \infty} ( - \log n  + \sum_{i=1}^n 1/i )$ is Euler's constant. The infinite sum converges both almost surely and in mean square, 
by Kolmogorov's theorem for sums of independent random variables with mean $0$, and the limit has the {\em Gumbel distribution function} $F(x) = \exp(-e^{-x})$.
More generally, the asymptotic behavior of the maximum $M_n$ of an i.i.d.\ sample from a continuous distribution is well understood.
If after proper rescaling, the distribution of $M_n$ has a non-degenerate weak limit, that limit must have the distribution function $F_\rho(x)=\exp(-(1+x\rho)^{-1/\rho})$, for 
some $\rho\in(-\infty,\infty)$
and $x$ such that $1+x\rho>0$ (and $F_\rho(x)$ equals 0 or 1 for other $x$), see, e.g.,~\cite{MR1791071}. Here $\rho$ (and the scaling) 
depends on the behavior of the distribution near the supremum of its support.  Limits can be also degenerate, and there exist distributions for which no non-degenerate limit is possible.

The situation is quite different for an infinite exchangeable sequence $X_1, X_2, \ldots$.
In this case any distribution can appear as a limiting distribution of the finite sample maximum $M_n:= X_{n:n}$, as shown by the following
example, which seems to be folklore. Let $Z_1,Z_2,\dots$ be a sequence of  i.i.d.\  random variables and let $M$ be a random variable,
independent of this sequence, with some given distribution. 
Take $X_n=Z_n+M$ to obtain an exchangeable sequence $X_1,X_2,\dots$.
If the support of the distribution of $Z_1$ is bounded above then $M_n$ converges a.s.\ to a shift of $M$,
 without any rescaling. 
So to obtain results of any interest about limit distributions of maxima from an exchangeable sequence, some further structure must be involved.

We are interested here in the distribution of sample gaps, sample extremes, and related statistics, 
for exchangeable samples from a random discrete distribution $\Pbul:= (P_1,P_2, \ldots)$ 
on the set $\BN: = \{1,2, \ldots\}$ of positive integers, subject to
\begin{equation}
\label{proper}
0 < P_j < 1 \mbox{ for every } j = 1,2, \ldots \mbox{ and } \sum_{j= 1}^\infty P_j = 1  \mbox{ almost surely. }
\end{equation}
We specify $\Pbul$ by the {\em residual allocation model} ($\RAM$) or {\em stick-breaking scheme} 
\cite{MR0010342}, %%, AUTHOR = {Halmos, Paul R.}, TITLE = {Random alms},
\cite{sawyer1985sampling} %%, title={A sampling theory for local selection}, author={Sawyer, Stanley and Hartl, Daniel},
\cite[\S 5]{MR1337249} %%, AUTHOR = {Pitman, Jim}, TITLE = {Exchangeable and partially exchangeable random partitions},
\begin{equation}
\label{stickbreak}
P_j:= H_j \prod_{i=1}^{j-1} ( 1 - H_i),\quad  \mbox{ with}
\end{equation}
\begin{equation}
\label{properhazards}
 0 < H_i < 1 \qquad \mbox{and} \qquad \prod_{i=1}^\infty (1- H_i) = 0 \mbox{  almost surely}.
\end{equation}
The random variables $H_i$ may be called {\em residual fractions}, {\em random discrete hazards}, or {\em factors}.
We use the term $\RAM$ to indicate that the $H_i$ are independent, but not necessarily that they are identically distributed.
But we also consider these models in the broader context of $H_{\bullet}$ subject to \eqref{properhazards},
corresponding to $\Pbul$ subject to \eqref{proper}, without any further dependence assumptions, 
which we call a {\em generalized residual allocation model} ($\GRAM$) \cite[\S 5]{MR1337249}.

%%Here and throughout the article, the subscript $\bullet$ is used to indicate an index ranging over the set $\BN:= \{1,2,\ldots\}$ of positive integers.
The case of i.i.d.\ factors $H_i$ has been extensively studied by Gnedin and coauthors 
\cite{MR2044594}, %%%%AUTHOR = {Gnedin, Alexander V.}, TITLE = {The {B}ernoulli sieve}, 
\cite{MR2735350}, %% , AUTHOR = {Gnedin, Alexander and Iksanov, Alexander and Marynych, Alexander}, TITLE = {The {B}ernoulli sieve: an overview},
who call this model the {\em Bernoulli sieve}.
Another $\RAM$ of particular interest, because of its {\em invariance under size-biased permutation} 
\cite{MR1387889} %%%, AUTHOR = {Pitman, Jim}, TITLE = {Random discrete di
 is the $\GEM$ model with parameters $(\alpha,\theta)$. 
In this model, $H_i$ has the beta$(1-\alpha, \theta + i \alpha)$ density on $(0,1)$ 
\begin{equation}
\label{eq:Y}
\frac{ \BP[H_i\in dx]}{dx}=\frac{ x^{-\alpha}(1-x)^{\theta+i\alpha-1}}{B(1-\alpha,\theta+i\alpha)}\qquad \qquad(0 < x < 1, i\in\BN),
\end{equation} 
where $0 \le \alpha < 1$ and $\theta>-\alpha$ are real parameters, and $B(\cdot,\cdot)$ is Euler's beta function.
The $\GEM$ hazard variables are i.i.d.\  only in the important special case $\alpha = 0$ covered by the following theorem:

\begin{theorem}
\label{thm:Mn}
Let $X_1,X_2,\dots$ be an exchangeable sequence obtained by i.i.d.\  sampling from the $\GEM(0,\theta)$ distribution \eqref{stickbreak} for i.i.d.\ beta$(1,\theta)$ hazards $H_i$.
For each fixed $n \ge 1$, the gaps $G_{i:n}$ between the order statistics of the sample, read from right to left, with $G_{n:n}+1:= \min_{1 \le i \le n} X_i$, are 
independent geometric$(i/(i+\theta))$ variables, with means $\theta/i$ for $1 \le i \le n$. 
\end{theorem}

As detailed in Section \ref{sec:apps}, this theorem contains most known results about $\GEM(0,\theta)$ samples,
including the well known sampling formulas for $\GEM(0,\theta)$, due to 
Ewens \cite{MR0325177}, %%AUTHOR = {Ewens, W. J.}, TITLE = {The sampling theory of selectively neutral alleles}, JOURNAL = {Theoret. Population Biology},
Antoniak \cite{MR0365969}, %%%AUTHOR = {Antoniak, Charles E.}, TITLE = {Mixtures of {D}irichlet processes with applications to
and Donnelly and Tavar\'e \cite{MR827330}. %%, AUTHOR = {Donnelly, Peter and Tavar{\'e}, Simon}, TITLE = {The ages of alleles and a coalescent},
It is also very close to recent studies  of random compositions derived from $\RAM$s with i.i.d.\ factors,
as we acknowledge further below.
Our simple description of $\GEM(0,\theta)$ gaps is hidden in these studies by different encodings of the 
values and their multiplicities in discrete random sampling, based on the {\em count sequence} $N^\circ_{\bullet:n}:= (N^\circ_{1:n}, N^\circ_{2:n}, \ldots)$  defined by
\begin{equation}
\label{boxcounts}
N^\circ_{b:n} := \sum_{i=1}^n \ind ( X_i = b) \qquad (b = 1,2, \ldots).
\end{equation}
Here $\ind(A)$  denotes the indicator of an event or set $A$.
The notion of a random sample from a random discrete distribution admits 
a variety of interpretations, some of which are recalled in Section \ref{sec:apps}.
But as our primary metaphor for sampling, we follow recent studies of the 
Bernoulli sieve \cite{MR2735350} %% , AUTHOR = {Gnedin, Alexander and Iksanov, Alexander and Marynych, Alexander}, TITLE = {The {B}ernoulli sieve: an overview},
in regarding the sample  $X_1, \ldots, X_n$ as an allocation of $n$ balls labeled by $i = 1,2, \ldots, n $ into an unlimited number of boxes 
labeled by $b \in \{1,2, \ldots\}$. 
So $X_i$ is the label of the box into which ball $i$ is thrown. Given $\Pbul$ the $X_i$ are independent allocations with 
$\BP(X_i = b \giv \Pbul) = P_b$.  The count $N^\circ_{b:n}$ is the number of balls thrown into box $b$, the sample maximum $X_{n:n} = \max\{b: N^\circ_{b:n} >0\}$ 
is the label of the rightmost occupied box, and so on.

The key to Theorem \ref{thm:Mn} is the close parallel between the structure of gaps in sampling from $\GEM(0,\theta)$, and from exponential$(\lambda)$, as in \eqref{explgaps}.
This parallel guided our choice to list the gaps from top down rather than bottom up, as well as the definition of the final gap $G_{n:n}:= X_{1:n} - 1$ in the discrete case.
We show in Section \ref{sec:pp} how Theorem \ref{thm:Mn} follows easily from its exponential$(\lambda)$ analog, using 
Ignatov's construction \cite{MR645134} %%%, AUTHOR = {Ignatov, Ts.}, TITLE = {A constant arising in the asymptotic theory of symmetric
of $\GEM(0,\theta)$ from 
a Poisson point process.  

This construction, and the change of variables \eqref{chvar}, which maps sampling by independent uniforms in $(0,1)$ to sampling by independent exponentials in $(0,\infty)$,
was developed and applied in a number of previous works 
\cite{MR645134}, \cite{MR2735350}, %% , AUTHOR = {Gnedin, Alexander and Iksanov, Alexander and Marynych, Alexander}, TITLE = {The {B}ernoulli sieve: an overview},
to deduce results for sampling from $\RAM$s and related regenerative composition structures from corresponding results in renewal theory. 
The method yields also the following corollary of Theorem \ref{thm:Mn}:

\begin{corollary} \label{char:indpt}
The $\GEM(0,\theta)$ models for\/ $0 < \theta < \infty$ are the only\/ $\RAM$s with i.i.d.\ factors 
such that for all sufficiently large $n$ the gaps between order statistics in a sample of size $n$ are independent. 
\end{corollary}

We conjecture that the $\GEM(0,\theta)$ models are the only random discrete distributions of any 
kind subject to \eqref{proper} with this property of independence of sample gaps for all $n$, or for all large $n$.  
But resolving this question seems beyond the reach of our current methods.

We discovered these properties of gaps in $\GEM(0,\theta)$ samples
by seeking an adequate explanation of the identity in distribution for the maximum $M_n$ of a $\GEM(0,\theta)$ sample,
presented in the next corollary of Theorem \ref{thm:Mn}.
We first found this identity by a different method indicated in Section \ref{sec:max}, without consideration of gaps.
But the gaps explain it much better:

\begin{corollary}
\label{crl:Max}
For the maximum $M_n$ of a sample of size $n$ from $\GEM(0,\theta)$ 
\begin{equation}
\label{eq:Mnassum}
M_n:= \max_{1 \le i \le n} X_i = \max \{b: N^\circ_{b:n} > 0\} = 1 + \sum_{i=1}^n G_{i:n} \ed 1 + \sum_{i=1}^n G_i 
\end{equation}
for independent $G_i$ with the geometric$(i/(i+ \theta))$ distribution.
\end{corollary}

Since $G_i$ has mean $\BE G_i=\theta/i$ and variance $\Var G_i = \theta ( i + \theta)/i^2 \sim \theta/i$ as $i\to\infty$,
Lindeberg's central limit theorem implies the limit in distribution
\begin{equation}
\label{eq:Mnclt}
\frac{M_n-\theta\log n}{\sqrt{\theta\log n}}\overset{d}{\to} Z\qquad \mbox{ as } n\to\infty,
\end{equation}
where $Z$ has the standard normal distribution. This limit theorem for $M_n$ is known 
as an instance of
a more general normal limit theorem for $M_n$ in sampling from a large class of $\RAM$s with i.i.d.\  factors \cite[Theorem 2.1 (b)]{MR2538083}. 
Also known \cite[(5.22)]{MR2032426} %% AUTHOR = {Arratia, Richard and Barbour, A. D. and Tavar{\'e}, Simon}, TITLE = {Logarithmic combinatorial structures: a probabilistic
and \cite[Theorem 2.3]{MR2538083} %%% AUTHOR = {Gnedin, Alexander V. and Iksanov, Alexander M. and Negadajlov, Pavlo and R{\"o}sler, Uwe},jt
is the fact that \eqref{eq:Mnclt} holds also with $M_n$ replaced by the number $K_n$ of distinct values in the sample,
which may expressed in various ways parallel to the expressions for $M_n$ in \eqref{eq:Mnassum}:
\begin{equation}
\label{kndef}
K_n:= \sum_{j = 1}^n \ind \bigl( X_j \notin \{X_1, \ldots, X_{j-1}\} \bigr) = \sum_{b=1}^{\infty} \ind(N^\circ_{b:n} >0 ) 
%%= 1  + \sum_{j= 2}^{n}  \ind( G_{j-1:n} > 0 ).
= 1  + \sum_{i= 1}^{n-1}  \ind( G_{i:n} > 0 ).
\end{equation}
%%% AUTHOR = {Gnedin, Alexander V. and Iksanov, Alexander M. and Negadajlov, Pavlo and R{\"o}sler, Uwe},
In Section \ref{sec:apps} we discuss further these different representations of $K_n$, their interpretations in various
applications, and related representations of the counts
%\begin{equation}
%\label{kjndef}
\[
K_{j:n}:= \sum_{b=1}^{\infty} \ind(N^\circ_{b:n} = j )   
\]%\end{equation}
which for $1 \le j \le n$ gives the numbers of clusters of size $j$ in the sample,  and for $j= 0$ gives 
\begin{equation}
\label{deltan}
K_{0:n} := M_n - K_n  = \sum_{j=1}^ {M_n} \prod_{i=1}^n \ind(   X_i \ne  j ) = G_{n:n} + \sum_{i=1}^{n-1} (G_{i:n} - 1)_+  
\end{equation}
which is the total count of all values between $1$ and $M_n$ that are missing in the sample of size $n$.
(Here and below $(x)_+=\tfrac12(x+|x|)$ is the positive part of $x$.)
There is by now a substantial theory of asymptotics for these and related statistics of samples from $\RAM$s with i.i.d.\  factors,
 developed by Gnedin and coauthors by various techniques.  
In particular, the work of \cite{MR2538083} %%% AUTHOR = {Gnedin, Alexander V. and Iksanov, Alexander M. and Negadajlov, Pavlo and R{\"o}sler, Uwe},jt
shows that the central limit theorems \eqref{eq:Mnclt} for $M_n$ and the same result for $K_n$
hold jointly with the same limit variable  $Z$, for the simple
reason that for a large class of $\RAM$s with i.i.d.\  factors, including $\GEM(0,\theta)$, the difference $M_n-K_n$ in \eqref{deltan}
has a limit in distribution as $n \to \infty$, without any centering or scaling. 
See \cite{MR2538083} %%% AUTHOR = {Gnedin, Alexander V. and Iksanov, Alexander M. and Negadajlov, Pavlo and R{\"o}sler, Uwe},jt
for further discussion, especially  \cite[Proposition~5.1]{MR2538083} %%% AUTHOR = {Gnedin, Alexander V. and Iksanov, Alexander M. and Negadajlov, Pavlo and R{\"o}sler, Uwe},jt
for a pretty formula involving the gamma function for the probability generating function
of the large $n$ limit in distribution of $K_{0:n}$ for $\GEM(0,\theta)$,
and \cite{MR2508790} %%% AUTHOR = {Gnedin, Alexander and Iksanov, Alex and Roesler, Uwe}, TITLE = {Small parts in the {B}ernoulli sieve}.
for generalizations to other $\RAM$s.
Other recent articles about refined limit theorems for various counting processes derived from $\RAM$s with i.i.d.\  factors are 
\cite{MR3176494} %, AUTHOR = {Iksanov, Alexander}, TITLE = {On the number of empty boxes in the {B}ernoulli sieve {I}},
\cite{MR2926172} %, AUTHOR = {Iksanov, Alexander}, TITLE = {On the number of empty boxes in the {B}ernoulli sieve {II}},
\cite{MR3416065} %, AUTHOR = {Iksanov, A. M. and Marynych, A. V. and Vatutin, V. A.}, TITLE = {Weak convergence of finite-dimensional distributions of the number of empty boxes in the {B}ernoulli sieve},
\cite{alsmeyer2016functional}.%%%, title={Functional limit theorems for the number of occupied boxes in the Bernoulli sieve},

We find the normal limit law for $M_n$ derived by sampling from a $\RAM$ with i.i.d.\  factors interesting, 
because normal limits cannot occur for the maximum of an i.i.d.\  sequence.
Compare with the quite different conclusion of \eqref{explim} for i.i.d.\  sampling from an exponential distribution,
where $M_n$ has a limit in law with just centering and no normalization. 
So, even though we regard Theorem \ref{thm:Mn} as a discrete analog of the more familiar description of gaps in i.i.d.\  sampling from an exponential distribution,
the limit theorems for $M_n$ implied by these results are quite different.
Naively, it might be expected that the discrete analog of the simple structure of gaps in an exponential$(\lambda)$ sample should
be the gaps in sampling from a geometric$(p)$ distribution, which is the $\RAM$ \eqref{stickbreak} with deterministic factors $H_i \equiv p$.
But apart from some simple results for a sample of size $n=2$, which are easily seen to hold for any $\RAM$ with $H_1$ independent of $H_2,H_3, \ldots$,
 the possibilities of various configurations of ties for $n \ge 3$ makes the description of gaps and related statistics for geometric$(p)$ sampling more complicated than might 
at first be expected. 
The distribution of the count of missing values $K_{0:n}$ in \eqref{deltan}, 
as well as the number $L_n$ of {\em ties with the maximal value} $M_n$
\begin{equation}
\label{nties}
L_n:= n - \max \{ j : X_{j:n} < X_{n:n} \}  = N^\circ_{{M_n}:n} =  \min \{ i : G_{i:n} > 0 \}
\end{equation}
have been the subject of many studies of i.i.d.\  sampling from fixed discrete distributions, 
especially from the geometric$(p)$ distribution.    
%%% with $H_i \equiv p$.
See for instance
\cite{MR2023876} %%%, AUTHOR = {Bruss, F. Thomas and Gr{\"u}bel, Rudolf}, TITLE = {On the multiplicity of the maximum in a discrete random sample},
\cite{MR2582705} %%%, AUTHOR = {Gr{\"u}bel, Rudolf and Hitczenko, Pawe{\l}}, TITLE = {Gaps in discrete random samples},
and earlier literature cited there.
In sampling from the geometric$(p)$ distribution it is known the distributions of $K_{0:n}$ and $L_n$ remain tight as $n \to \infty$, but that they do not converge, due to a periodic phenomenon 
which has been extensively studied in this literature. 
%To provide a context including both the geometric and $\GEM$ cases, in Section  \ref{xx} 
%we describe the joint distribution of gaps in sampling from a quite general random discrete distribution, 
%which is computationally manageable at least for any $\RAM$ with i.i.d.\  factors.
%This leads to a more computational proof of Theorem \ref{xx}, without consideration of Poisson processes or exponential variables.
As shown in \cite{MR2538083} however, %%% AUTHOR = {Gnedin, Alexander V. and Iksanov, Alexander M. and Negadajlov, Pavlo and R{\"o}sler, Uwe},jt 
for a large class of $\RAM$s including $\GEM(0,\theta)$, the periodic phenomena which arise from i.i.d.\   geometric$(p)$ sampling get smoothed out very nicely:
the count $L_n$ and related statistics such as the $K_{j:n}$ have limits in distribution without any centering or normalization as $n \to \infty$.
The idea that such results should be informed by analysis of gaps as well as counts leads to some developments of those limit theorems 
explored in a sequel \cite{ramgaps} of this article. 

For $\RAM$s with independent but non-identically distributed factors, our results are more limited.
For the two parameter $\GEM(\alpha,\theta)$ distribution the representation \eqref{eq:Mnassum} of $M_n$ as a sum of independent random variables is
no longer valid. Nevertheless we show in Section~\ref{sec:max2} that in this case the maximum of a size $n$ sample behaves as a random multiple of 
$n^{\alpha/(1-\alpha)}$ as $n\to\infty$, and in Section~\ref{sec:tiesatmax} we indicate some companion limit theorems for $L_n$ in this case.
See also a sequel to this article \cite{gibbs}, where for sampling from $\GEM(\alpha,\theta)$
we derive a result presented here without proof as Theorem \ref{thm:gibbs}, which gives the distribution of 
the value-ranked frequencies, meaning the sequence of non-zero components of $(N^\circ_{b:n} , b = 1, 2, \ldots )$.

%We also indicate how our search for extensions of Theorem \ref{thm:Mn} to more general $\RAM$s raises some 
%interesting questions about the two-parameter $\GEM$ which are only partially answered here.

%Throughout the paper we denote by $\ind(A)$ of 
%the indicator of the set (or event) $A$. 
% For two sequences of random variables we write $A_n\simas B_n$ as $n\to\infty$
%if the limit $A_n/B_n$ exists and equals 1 a.s. % The set of natural numbers is denoted $\BN=\{1,2,\dots\}$.

\section{Point processes associated with random discrete distributions}
\label{sec:pp}

For a random discrete distribution $\Pbul:= (P_j)$ on the positive integers governed by the $\GRAM$ \eqref{stickbreak} let
\[%begin{equation}\label{breakpoints}
\YY_0:= 0 \mbox{ and } \YY_j:= \sum_{i=1}^j P_i = 1 - \prod_{i=1}^j (1 - H_i) \mbox{ for } k = 1,2, \ldots.
\]%end{equation}
Thinking of $\Pbul$ as a random discrete distribution on the real line, which happens to be concentrated on positive integers, $F_j = \Pbul(-\infty, j]$ is the random cumulative
distribution function evaluated at positive integers $j$.  In the stick-breaking interpretation, the $\YY_j \in [0,1]$ are the {\em break points}.  But we prefer the language of the
{\em stars and bars model}, discussed further in \cite{ramgaps}. 
Following the method developed by 
Ignatov \cite{MR645134} %%%, AUTHOR = {Ignatov, Ts.}, TITLE = {A constant arising in the asymptotic theory of symmetric
for $\GEM(0,\theta)$,  and further developed in 
\cite{MR2122798} %%, AUTHOR = {Gnedin, Alexander and Pitman, Jim}, Regen comps
\cite{MR2183216} %%%AUTHOR = {Gnedin, A. and Pitman, J.}, TITLE = {Self-similar and {M}arkov composition structures},
\cite{MR2351686} %%%AUTHOR = {Gnedin, Alexander and Pitman, Jim}, TITLE = {Poisson representation of a {E}wens fragmentation process},
for various other models of random discrete distributions, we treat the {\em bars} $\YY_j$ 
as the points of a simple point process  $N_F$
on $(0,1)$, which counts bars not including endpoints of $[0,1]$.  So
\[%begin{equation}\label{NFdef}
N_F(a,b]:= \sum_{k=1}^\infty \ind( a < \YY_k  \le b )   \qquad ( 0 \le a < b < 1 )
\]%end{equation}
is the number of bars in $(a,b]$. The {\em stars} are the i.i.d.\ uniform on $[0,1]$ points 
$U_1,U_2,\ldots$ which fall between bars and define a 
random sample $X_1,X_2, \ldots$  from $\Pbul$ as $X_i=N_F(0,U_i]$. Then
\[%begin{equation}\label{breakcdf}
\BP( X_i \le j \giv \Pbul ) =  \YY_j \qquad (i = 1,2, \ldots,\, j = 1,2, \ldots).
\]%end{equation}
In terms of Gnedin's {\em balls in boxes model} of \cite{MR2044594}, %%%%AUTHOR = {Gnedin, Alexander V.}, TITLE = {The {B}ernoulli sieve}, 
the bars $F_j$ in $(0,1)$ are the dividers between boxes $[F_{j-1}, F_j)$ which collect the stars (balls) $U_i$ into clusters falling in the same box.
For the $\YY_j$, the prevailing assumption \eqref{proper} becomes 
\[%begin{equation}
\label{properys}
0 < \YY_1 < \YY_2 < \cdots \uparrow 1  \qquad \mbox{ almost surely.}
\]%end{equation}

It is convenient to make the change of variable from $[0,1)$ to $[0,\infty)$ by the map $u\mapsto x = - \log(1-u)$. 
This is the inverse of the cumulative distribution function $x \mapsto 1 - e^{-x}$ of a standard exponential variable 
$\eps := - \log (1-U)$ for $U$ uniform on $(0,1)$.
Let  $S_0:= 0$ and
$$
S_j:= - \log ( 1 - \YY_j) = \sum_{i=1}^j - \log(1 - H_i) \qquad ( j = 1,2, \ldots).
$$
We regards these images $S_j$ of bars $\YY_j$ as the points of an associated point process $N_S$ on $(0,\infty)$:
\[
\label{NSdef}
N_S(s,t]:= \sum_{k=1}^\infty \ind( s < F_k  \le t )=N_F(1-\re^{-s},1-\re^{-t}]   \qquad ( 0 \le s < t < \infty ).
\]
Note that the assumption \eqref{proper} translates into $0 < S_1 < S_2 < \cdots \uparrow \infty$ a.s.
For convenience we recall the Ignatov's construction of $\GEM(0,\theta)$.

\begin{lemma}[\cite{MR645134}]\label{lem:poisson} With above notation, the following
conditions are equivalent:
\begin{itemize}
\item[\upshape(i)] $\Pbul$ has the $\GEM(0,\theta)$ distribution, meaning the $H_i$ are i.i.d.\ {\upshape beta}$(1,\theta)$.

\item[\upshape(ii)] $N_F$ is an inhomogeneous Poisson process on $(0, 1)$ with intensity $\theta\, du/(1 - u)$ at %\linebreak[4]
$u\in(0, 1)$.

\item[\upshape(iii)] $N_S$ is a homogeneous Poisson process on $(0,\infty)$ with intensity $\theta\,dt$ at $t > 0$.

\item[\upshape(iv)] The scaled spacings $(S_k - S_{k-1})/\theta$, $k = 1, 2,\dots$ are i.i.d.\ {\upshape exponential}$(1)$.
\end{itemize}
\end{lemma}

\noindent
\begin{proof}[Proof of Theorem \ref{thm:Mn}]
Whatever the random discrete distribution $\Pbul$ subject to \eqref{proper}, for the sample $(X_1, \ldots, X_n)$ constructed as above, the
order statistics of the discrete sample are obtained by counting bars to the left of the order statistics of the corresponding uniform
and exponential samples, according to  the formula
\begin{equation}
\label{xue}
X_{i:n} - 1 = N_F( 0, U_{i:n}]  =  N_S( 0, \eps_{i:n} ] ,
\end{equation}
while the gaps between order statistics of the discrete sample are obtained by counting bars between corresponding order statistics of the uniform
and exponential samples:
\begin{equation}
\label{gue}
G_{i:n} = N_F( U_{n-i:n}, U_{n-i +1:n} ] =  N_S( \eps_{n-i:n}, \eps_{n-i +1:n} ] .
\end{equation}
%According to a result of Ignatov \cite{MR645134} %%% , AUTHOR = {Ignatov, Ts.}, TITLE = {A constant arising in the asymptotic theory of symmetric
%the point process $N_S$ derived from $\GEM(0,\theta)$ is homogeneous Poisson with rate $\theta$.
According to Lemma~\ref{lem:poisson} the point process $N_S$ is homogeneous Poisson with rate $\theta$. By construction it is independent of the gaps between exponential order statistics appearing in \eqref{gue}, which due to \eqref{explgaps} 
are independent and distributed like $\eps_i/i$ for $1 \le i \le n$.
It follows that the $G_{i:n}$ are independent, with $G_{i:n} \ed N_S(0,\eps/i]$ for $\eps$ standard exponential independent of $N_S$.
So the distribution of $G_{i:n}$ is the mixture of Poisson$(\theta t)$ distributions with $t$ assigned the exponential$(i)$ distribution of $\eps/i$.
It is well known that such a mixture is geometric$(p)$ for $p$ determined by equating means: $(1-p)/p = \theta/i$. This gives $p = i/(i+\theta)$, and the conclusion follows.
\end{proof}

\begin{proof}[Proof of Corollary \ref{char:indpt}]
%{\tt xxx Yuri please check this is OK!}
For a RAM with i.i.d.\  factors the point process $N_S$ is a renewal process.  
For each fixed $\delta$ with $0 < \delta < 1$,  the number of points $\YY_j$ such that
$\YY_j \le 1-\delta$  is $N_S(0, - \log(\delta)]$, which is well known to have finite exponential moments.
On the other hand, by the law of large numbers for the sampling process, provided $- \log (\delta)$ is chosen to be  a continuity
point of the renewal measure, 
%{\tt xxx this is the only tricky point, hopefully OK as written!}
with probability one, for sufficiently large $n$, both the sum of gaps $\sum_{j= \delta n} ^n G_{j:n}$ and the
sum of indicators of non-zero gaps $\sum_{j= \delta n }^n \ind( G_{j:n} >0)$ will eventually equal the number of renewals $N_S(0, - \log(\delta)]$,
because every relevant box will be occupied. Assuming the gaps are independent, it then follows from the Poisson 
approximation to sums of independent Bernoulli$(p_i)$ random variables with small parameters $p_i$ that $N_S(0, - \log(\delta)]$ is Poisson distributed. A variation of this argument shows that
$N_S$ has independent, Poisson distributed increments. In other words, $N_S$ is a possibly inhomogeneous Poisson process. But the $\RAM $ assumption makes $N_S$ a renewal process.
By comparison of the Poisson and renewal decompositions of $N_S$ at its first point $S_1$, it is easily shown that the point processes that are both Poisson processes and renewal processes
are the Poisson processes of constant rate $\theta$ for some $\theta > 0$.
So the conclusion follows from 
Ignatov's construction of $\GEM(0,\theta)$ in Lemma~\ref{lem:poisson}.
\end{proof}

As shown in \cite{MR2122798} %%, AUTHOR = {Gnedin, Alexander and Pitman, Jim}, Regen comps
\cite{MR2508790} %%% AUTHOR = {Gnedin, Alexander and Iksanov, Alex and Roesler, Uwe}, TITLE = {Small parts in the {B}ernoulli sieve}.
\cite{MR2538083} %%% AUTHOR = {Gnedin, Alexander V. and Iksanov, Alexander M. and Negadajlov, Pavlo and R{\"o}sler, Uwe}
an asymptotic analysis of sampling statistics for a $\RAM$ with i.i.d.\  factors can be made by exploiting properties of the renewal counting process $N_S$ in this case.
But for $\GEM(\alpha,\theta)$ with $0 < \alpha < 1$, the spacings between the $S_j$ are independent but not identically distributed, 
with $S_j - S_{j-1}$ converging almost surely to $0$ as $j \to \infty$, by a simple Borel--Cantelli argument.
%{\tt xxx Yuri check you agree! }
It follows that in sampling from $\GEM(\alpha,\theta)$ for $0 < \alpha < 1$ the behavior of the order statistics and gaps 
is very different from the case $\alpha = 0$, and not approachable by methods of renewal theory.

Both to illustrate applications of Theorem \ref{thm:Mn} in the case $\alpha = 0$, and to motivate study of the $\GEM$ order statistics 
for $0 < \alpha < 1$, before proceeding with that study we first show how Theorem~\ref{thm:Mn} leads to some known results
about the $\GEM$ model, and recall some of its diverse applications.

\section{Applications}
\label{sec:apps}
The $\GEM$ acronym was assigned by Warren Ewens \cite[p. 321]{MR2026891} %%%, AUTHOR = {Ewens, Warren J.}, TITLE = {Mathematical population genetics. {I}}, 
to acknowledge the work of 
Griffiths \cite{griffiths80}, 
Engen \cite{MR0411097} %%%, AUTHOR = {Engen, Steiner}, TITLE = {A note on the geometric series as a species frequency model},
\cite{MR515721}
and
McCloskey \cite{mccloskey1964model} %%%title={MODEL FOR DISTRIBUTION OF INDIVIDUALS BY SPECIES IN ENVIRONMENT}, author={McCloskey, John William},
in
developing the $\GEM$ model for random frequencies in genetics and ecology.
The $\GEM(0,\theta)$ with i.i.d.\  beta$(1,\theta)$ factors was studied first, following which the two parametric extension proposed by 
Engen \cite{MR515721} 
has also been 
extensively studied \cite{MR1156448,MR2245368,MR2663265}, 
motivated by its appearance in the structure of interval partitions generated by the zeros of various stochastic processes. 
New models of stationary reversible dynamics for population frequencies consistent 
with $\GEM(\alpha,\theta)$  for $0 < \alpha < 1$
have recently been developed and are of continuing interest 
\cite{MR2596654} %%AUTHOR = {Petrov, L. A.}, TITLE = {A two-parameter family of infinite-dimensional diffusions on the {K}ingman simplex},
\cite{costantini2016wright}. %%, title={Wright-Fisher construction of the two-parameter Poisson-Dirichlet diffusion}, author={Costantini, Cristina and De Blasi, Pierpaolo and Ethier, Stewart N and Ruggiero, Matteo and Spano, Dario},
The $\GEM$ model has also been applied in Bayesian non-parametric statistics
as a building block for Bayesian analysis and machine learning algorithms for inferences 
about clustered data.
See the recent review by Crane \cite{MR3458585} \cite{MR3458591}. 

\subsection{Species sampling}
\label{sec:species}

The $\GEM$ distributions were first studied in the setting of ecology, where a random discrete distribution $\Pbul$ %%:= (P_1, P_2, \ldots)$ 
may be regarded as a {\em species abundance model},
meaning an idealized  listing of frequencies of an unlimited number of distinct  species  in a population of unlimited size.
To justify the use of infinite models in this setting,  some preliminary remarks may be in order.
All actual populations are finite, with only a finite number of species.  However ecologists discovered that while models with a large
finite number of species are quite intractable, remarkable simplifications occur in some particular infinite models.
In a sample of $n$ individuals from a large population of size say $\Nbig$, 
one can suppose that size $\Nbig$ population itself is a sample from some ideal infinite population. Finite samples from
this population may be assumed to exchangeable, and de Finetti's theorem leads to a representation of consistent models of species sampling from a large
population in terms of limiting random frequencies, as shown by Kingman.

In the context of species sampling, in a sample of size $n$ from some population, 
the data acquired is most naturally regarded as just the partition of $n$, that is a collection of positive integers that sum to $n$,
obtained by classifying individuals by species, and ignoring any species labels.
Such a partition of $n$ can be described in two different ways. The first is to list the number $N_{i:n}$ of individuals of type $i$, for
$1 \le i \le k$, where $k=K_n$ is the number of distinct species in the sample, and there is some convention for the ordering of types.
The second is to
provide for each $j\in[n]$ the count
\begin{equation}
\label{kjn}
K_{j:n}:= \sum_{i=1}^n \ind( N_{i:n} = j )  \qquad (j = 1, 2, \ldots )
\end{equation}
of species with $j$ representatives in the size $n$ sample.
The total number of distinct species is then
\[%begin{equation}\label{knpermdef}
K_n =  \sum_{j=1}^n K_{j:n}\, .
\]%end{equation}
The frequencies $P_i$ of species in the ideal infinite population then arise as almost sure limits of $N_{i:n}/n$ as $n\to\infty$,
for some consistently chosen ordering of $N_{\bullet:n}$.

It is an awkward aspect of random frequency models that most labeling of frequencies $P_i$ by integers or other countable sets are somewhat arbitrary.
There are several possible workarounds for this problem. The easiest way is to list the {\em ranked sample frequencies}, meaning the numbers of representatives of various
species in descending order as
\[
\Ndec_{\bullet:n}=(\Ndec_{1:n},\dots,\Ndec_{k:n}),\qquad \Ndec_{1:n}\ge\dots\ge\Ndec_{k:n}>0,
\]
which is the decreasing rearrangement of any other listing of the sample frequencies $(N_{1:n},\dots,N_{k:n})$. This is also a common way to list parts of partitions in combinatorics.
Another way to arrange species is to obtain a size $n$ sample by sampling the population one by one and listing the species in order
of their appearance in the sampling process. Given that any particular species in the whole population has $m$ representatives
in a sample of size $n$, the probability of that species being listed first in order of appearance is $m/n$, by the assumed exchangeability of the sampling process.
This leads to the general notion of a {\em size-biased random permutation} of a
finite or countably infinite index set $I$, or of a collection of components of some kind $C_i, i \in I$ that is indexed by $I$,
for some notion of sizes $V_i:= V(C_i)$ of the components being permuted
\cite{MR1104569} %% , AUTHOR = {Donnelly, Peter}, TITLE = {The heaps process, li
\cite{MR1387889}.  %%, AUTHOR = {Pitman, Jim}, TITLE = {Random discrete distributions invariant under size-biased permutation},
Typically $V(C_i)$ is the number of elements for a finite set $C_i$, or 
some measure such as length for infinite sets $C_i$ like intervals.
The size function $V$ of components is subject to the requirement that $V_i >0$ and that $\Sigma:= \sum_{i} V_i  < \infty$, which needs to hold almost surely for a
collection of random components $(C_i, i \in I)$. Given such random components $(C_i, i \in I)$,
their {\em size-biased permutation}  is a random indexing of these components $(C_{\sigma(i)}, i = 1,2, \ldots)$, indexed  either by the set $[k]:= \{1,2, \ldots, k\}$ 
if there are a finite number $k$ of components,
or by $\BN$ if there are an infinite number of them. It is defined by a random bijection $\sigma$ from either $[k]$ or $\BN$ to $I$, such that
$\BP[\sigma(1)=j \,|\,C_i, i \in I ) =V_j/\Sigma$ for $j \in I$, when $C_{\sigma(1)}$ is called an {\em size-biased pick} from the components,
and for each $m \ge 1$ and $j_1,\dots,j_m$ with $\Sigma_m:= \Sigma -V_{j_1}-\dots- V_{j_m} >0$, the next component $C_{\sigma(m+1)}$ is an size-biased
pick from the remaining components indexed by $I \setminus\{j_1,\dots,j_m\}$:
$$\BP[\sigma(m+1)=j| C_i, i \in I; \sigma(1)=j_1,\dots,\sigma(m)=j_m]=V_j/\Sigma_m \mbox{  for all } j\in\ I \setminus\{j_1,\dots,j_m\}. $$

With component $C_i$ being a non-empty set of representatives of some species type $i$ in an exchangeable random sample of size $n$, 
for some arbitrary labeling of species by $i \in [k]$,
the sequence of sample frequencies {\em in order of appearance}
\[
\Nstar_{\bullet:n}=(\Nstar_{1:n},\dots,\Nstar_{k:n})
\]
is found to be the sequence of sizes $\Nstar_{i:n}=V(C_{\sigma(i)})$ of a size-biased random permutation of the clusters of different species found in the sample,
with size $V(C_i)=N_{i:n}$ being the number of each species present in the sample.
The notation $\Nstar_{\bullet:n}$ is a mnemonic for this size-biasing which is involved in any ordering of species by appearance in an exchangeable process of random sampling.  
This size-biased listing of sample cluster sizes in order of appearance turns out to be very convenient to work with,
especially when the frequencies $P_\bullet$ are in a size-biased random order themselves, an assumption which is quite natural in this context.  
This assumption holds for $\GEM(\alpha,\theta)$ model, as shown in  \cite{MR1156448}, %% AUTHOR = {Perman, Mihael and Pitman, Jim and Yor, Marc},
which is one of the reasons for use and study of this model.

In the context of species sampling, the values $X_i$ in
a size $n$ sample from a random discrete distribution such as $\GEM(\alpha,\theta)$ 
seem to be of little interest besides the way that clusters of these values 
define a partition or a composition of $n$.  However a natural meaning for these sample values $X_i$ and their order statistics $X_{\bullet:n}$ 
can be provided using the fact that $\GEM(\alpha,\theta)$ 
frequencies $\Pbul$  are already in size-biased random order, hence invariant in distribution under size-biased permutation ($\ISBP$) 
\cite{MR1387889},  %%, AUTHOR = {Pitman, Jim}, TITLE = {Random discrete distributions invariant under size-biased permutation},
meaning that
\[%begin{equation}\label{isbp}
 (P^*_1, P^*_2, \ldots) \ed  (P_1, P_2, \ldots )
\]%end{equation}
where $(P^*_1, P^*_2, \ldots)$ is a size-biased permutation of $(P_1, P_2, \ldots )$.

\begin{proposition}
\label{prop:species}
Consider an exchangeable process of species sampling $(Y_1, Y_2, \ldots)$ from random frequencies $P_\bullet$.
Split $(Y_1, Y_2, \ldots)$ into an {\em initial sample} $(Y_1, \ldots, Y_n)$ of size $n$, 
followed by a {\em secondary sample} $(Y_{n+1}, Y_{n+2}, \ldots)$ of unlimited size. 
For each $1 \le i \le n$ let $X_i$ be the number of species discovered by the secondary sample up to and including discovery
of $Y_i$ in the secondary sample.
Then $(X_1, \ldots, X_n)$ is a sample of size $n$ from $P_\bullet^*$, a size-biased random permutation of $P_\bullet$. 
In particular if $P_\bullet$ is $\ISBP$, as is the case for $\GEM(\alpha,\theta)$,
then $(X_1, \ldots, X_n)$ is distributed like a sample of size $n$ from $P_\bullet$.
Moreover, for a sample $(X_1, \ldots, X_n)$ from $P_\bullet^*$ so constructed,
as well as the usual interpretation of $K_{j:n}$ as numbers of species with $j$ representatives in the initial sample, and $K_n = \sum_{j=1}^n K_{j:n}$ the
total number of species found by the initial sample, various features of the order statistics in the sample can be interpreted as follows:
\begin{itemize}
\item the minimum sample value $X_{1:n}$ is the number of species encountered in the secondary sample up to and including when the first species in the primary sample is encountered.
\item the maximum sample value $M_n:= X_{n:n}$ is the number of species encountered in the secondary sample up to and including the first time $\tau_n$ that a member of each of the initial $K_n$ species has been encountered  in the secondary sample.
\item  the number $K_{0:n}:= M_n - K_n$ of missing values below the maximum of the sample,
is the  number of new species, not present in the initial sample,  which are encountered in the secondary sample before this stopping time $\tau_n$.
\end{itemize}
\end{proposition}

\begin{proof}
More formally, $X_i = x$ iff $i \in C(x)$, where
$C(1), C(2), \ldots$ is the list of clusters of individuals by species in the combined process, in their order of appearance in the secondary sample.
So for instance $C(1) := \{i \ge 1: Y_i = Y_{N(1)} \}$ with $N(1):= n+1$,  $C(2) := \{ i \ge 1: Y_i = Y_{N(2)} \}$, where $N(2) > N(1)$ is the index of the first individual of the second species to appear in the secondary sample, and so on.
The relabeling of species by their order of appearance in the secondary sample yields a size-biased permutation $P_\bullet^*$ of $\Pbul$
with $P_x^*$ the almost sure limiting relative frequency of $C(x)$ for each $x = 1,2, \ldots$.
It then follows by exchangeability that $X_1, \ldots, X_n$ is a sample of size $n$ from $P^*_\bullet$. 
The interpretations of the various statistics are then straightforward.
\end{proof}

One can also give similar interpretations of the gaps $G_{\bullet:n}$,  but this is a bit trickier. For a sample $X_1, \ldots, X_n$
let $\Nright_{\bullet:n}$ denote the list of sample frequencies in \textit{increasing order of $X$ values}, 
that is the subsequence of all strictly positive counts in the complete list of
counts $N^\circ_{\bullet:n}$ derived from the sample $X_1, \ldots, X_n$ as in \eqref{boxcounts}. 
So the count of values equal to $X_{1:n}$ comes first, and the count for $X_{n:n}$ last. 
In any random sample $X_1, \ldots, X_n$, it is clear from the definition that for any given composition $(n_1, \ldots, n_k)$ of $n$,
\[
\Nright_{\bullet:n}=(n_1,\dots,n_k)
\]
if and only if the sequence of gaps $G_{\bullet:n}$ between order statistics is such that
\begin{equation}
\label{nonzerogaps}
\{1\le i\le n-1:G_{i:n}>0\}\cup\{n\}=\{n_k,n_k+n_{k-1},n_k+n_{k-1}+n_{k-2},\dots,n\}.
\end{equation}
In the species sampling setting of Proposition \ref{prop:species}, 
$X_i$ is the number of species discovered by the secondary sample up to and including discovery
of the species of the $i$th initial individual, and $\Nright_{\bullet:n}$ is the
listing of cluster sizes in the primary sample in \textit{order of their discovery} by the secondary sample.
In this setting the gaps can be described as follows:
\begin{itemize}
\item For $\ell\in\{2,\dots,k\}$,  $G_{n_\ell+\dots+n_k:n}$ is the number of new species encountered in the secondary
sample after $\ell-1$ species of the primary sample are found up to and including the time when $\ell$-th species from the primary sample is found.
\item $G_{n:n}:=X_{1:n}-1$ is the number of species encountered in the secondary sample before some species present in the primary sample is found.
\item All other gaps are zero according to \eqref{nonzerogaps}.
\end{itemize}

\subsection{Some corollaries of  Theorem \ref{thm:Mn}}
\label{sec:corollaries}

We now explain how Theorem \ref{thm:Mn} implies a number of known results about $\GEM(0,\theta)$ samples.
Most of these were first discovered in the context of population genetics, where the age-ordering of alleles in a large population provides 
a natural indexing of allelic types, and the age-ordering of clusters of alleles found in a sample is of interest.  The first result is an easy corollary of Theorem \ref{thm:Mn}:

\begin{corollary}
\label{crlindics}
{\em (Gnedin-Pitman \cite[(3.1)]{MR2351686})}  %%%, AUTHOR = {Gnedin, Alexander and Pitman, Jim}, TITLE = {Poisson representation of a {E}wens fragmentation process},
In sampling from $\GEM(0,\theta)$, the sequence of indicators  of strictly positive gaps $\ind(G_{i:n} > 0)$ for $1 \le i \le n$ has the same distribution as the initial segment of
$n$ trials in an unlimited sequence $(B_1, B_2, \ldots )$ of independent Bernoulli trials with  $\BP(B_i =1) = \theta/(i + \theta)$:
\begin{equation}
\label{indicsident}
(\ind(G_{i:n} > 0), 1 \le i \le n) \ed (B_1, B_2, \ldots, B_n) .
\end{equation}
\end{corollary}

See also the work of Arratia, Barbour and Tavar\'e
\cite{MR1177897}  %%AUTHOR = {Arratia, Richard and Barbour, A. D. and Tavar{\'e}, Simon}, TITLE = {Poisson process approximations for the {E}wens sampling formula},
\cite{MR2032426}  %% AUTHOR = {Arratia, Richard and Barbour, A. D. and Tavar{\'e}, Simon}, TITLE = {Logarithmic combinatorial structures: a probabilistic 
\cite{MR2195574} %%%AUTHOR = {Arratia, Richard and Barbour, A. D. and Tavar{\'e}, Simon}, TITLE = {A tale of three couplings: {P}oisson-{D}irichlet and {GEM}
\cite{MR3458588}. %%% AUTHOR = {Arratia, Richard and Barbour, A. D. and Tavar{\'e}, Simon}, TITLE = {Exploiting the {F}eller coupling for the {E}wens sampling
for further study of this process of independent Bernoulli trials $(B_1, B_2, \ldots )$ and its relation to the Ewens sampling formula \eqref{esf} below.

Consider now the three random compositions of $n$ defined above in terms
of a sample $X_1, \ldots, X_n$ from a  random discrete distribution $\Pbul$:
\begin{itemize}
\item the value-ordered sample frequencies $\Nright_{\bullet:n}$, 
\item the appearance-ordered sample frequencies $\Nstar_{\bullet:n}$, and 
\item the ranked sample frequencies$\Ndec_{\bullet:n}$, which are the decreasing rearrangement of either $\Nright_{\bullet:n}$ or of $\Nstar_{\bullet:n}$;
\end{itemize}
For $\Pbul$ with $\GEM(0,\theta)$ distribution,
from \eqref{nonzerogaps} and \eqref{indicsident} it is quite easy to deduce the 
{\em Donnelly--Tavar\'e sampling formula} \cite{MR827330} %% AUTHOR = {Donnelly, Peter and Tavar{\'e}, Simon},  TITLE = {The ages of alleles and a coalescent} 
for the value-ordered frequencies:
%The expression \eqref{dtnright} is also well known
%for the probability $\BP_{0,\theta}(\Nstar_{\bullet:n}=(n_1,\dots,n_k))$ of the composition in order of appearance. 
\begin{equation}
\label{dtnright}
\BP_{0,\theta}(\Nright_{\bullet:n}=(n_1,\dots,n_k))=\frac{n!\,\theta^k}{(\theta)_n}\,\prod_{i=1}^{k}\frac{1}{n_i+\dots+n_k}
\end{equation}
for every composition $(n_1, \ldots , n_k)$ of $n$ into $k \le n$ parts, with $(x)_n:= \Gamma(x + n)/\Gamma(x)$ the Pochhammer symbol. 
The subscript notation $\BP_{\alpha,\theta}$ or $\BE_{\alpha,\theta}$ signals 
that probabilities or expectations are governed by the $\GEM(\alpha,\theta)$ model.
As indicated by Donnelly and Tavar\'e,
summing \eqref{dtnright}  over all compositions of $n$ with a prescribed weakly decreasing rearrangement 
yields  the celebrated {\em Ewens sampling formula}
\cite{MR0365969} %%AUTHOR = {Antoniak, Charles E.}, TITLE = {Mixtures of {D}irichlet processes with applications to
\cite{MR0325177} %%AUTHOR = {Ewens, Warren J.}, TITLE = {The sampling theory of selectively neutral alleles},
for the distribution of the partition of $n$ generated
by sampling from $\GEM(0,\theta)$. 
That is, for  $K_{j:n}$ 
the count of clusters of size $j$  as in \eqref{kjn},
%number of parts of size $j$ in the composition of $n$ generated by the sample, as above, 
for each weak composition $(m_1, \ldots, m_n)$ of $k$, meaning $m_j \ge 0$ with $\sum_{j=1}^n  m_j = k$, with $\sum_{j=1}^n j m_j = n$
\begin{equation}
\label{esf}
\BP_{0,\theta}(K_{j:n} = m_j, 1 \le j \le n) =   \frac{n!\, \theta^k } {(\theta)_n } \prod_{j = 1}^ n \frac{1}{m_j!  j ^{m_j}  } \,.
\end{equation}
It is also easily seen from \eqref{dtnright} and \eqref{esf}
that the composition probability function displayed in \eqref{dtnright} is also the
composition probability function of the appearance-ordered frequencies $\Nstar_{\bullet:n}$. Recalling from earlier discussion, that in sampling from any $\Pbul$,
\[%begin{equation}\label{sbstar}
\mbox{$\Nstar_{\bullet:n}$ is a size-biased permutation of $\Nright_{\bullet:n}$.} %%%%, as well as of $\Ndec_{\bullet:n}$,
\]%end{equation}
%%%, the size-biased permutation of either $\Nright_{\bullet:n}$ or $\Ndec_{\bullet:n}$.
the Donnelly--Tavar\'e formula implies 
the following very special property of $\GEM(0,\theta)$, in which the roles of $\Nstar_{\bullet:n}$ and $\Nright_{\bullet:n}$ can be reversed.

\begin{corollary}[{Donnelly and Tavar\'e \cite[(4.4)]{MR827330}}] %%, AUTHOR = {Donnelly, Peter and Tavar{\'e}, Simon}, TITLE = {The ages of alleles and a coalescent},
\label{crldt}
In a sampling from  $\GEM(0,\theta)$, the sample frequencies in value-order $\Nright_{\bullet:n}$ and 
in appearance-order $\Nstar_{\bullet:n}$ are identically distributed. 
Their common distribution is described by formula \eqref{dtnright}.  Consequently, in sampling from $\GEM(0,\theta)$, 
\[%begin{equation}\label{sbmagic}
\mbox{ $\Nright_{\bullet:n}$ is a size-biased permutation of $\Nstar_{\bullet:n}$.} 
%%%, as well as of $\Ndec_{\bullet:n}$. }
\]%end{equation}
\end{corollary}

The question of how to extend this result to $\GEM(\alpha,\theta)$  for $0 < \alpha < 1$ led us 
combine the representation of sampling from $\GEM(\alpha,\theta)$ provided by Proposition \ref{prop:species}
with the known description of the distribution of $\Nstar_{\bullet:n}$ for $\GEM(\alpha,\theta)$ 
\cite{MR1337249}  %    AUTHOR = {Pitman, Jim}, %     TITLE = {Exchangeable and partially exchangeable random partitions},
\cite[\S 3.2]{MR2245368},  %%% CSP 
to obtain the following theorem, whose proof will be detailed elsewhere \cite{gibbs}.

\begin{theorem}
\label{thm:gibbs}
In sampling from  $\GEM(\alpha,\theta)$, for all\/ $0 \le \alpha < 1$ 
\begin{equation}
\label{sbmagic:alpha}
\mbox{ $\Nright_{\bullet:n}$ is a (size$-\alpha)$-biased permutation of $\Nstar_{\bullet:n}$.}
%%%%, as well as of $\Ndec_{\bullet:n}$. }
\end{equation}
\end{theorem}
The meaning  of \eqref{sbmagic:alpha} is that given $\Nstar_{\bullet:n} = (n_1, \ldots, n_k)$, the frequency $\Nright_{1:n}$ of the minimal value is distributed like a random choice of
$(n_1, \ldots, n_k)$, with $n_i$ chosen with probability $(n_i - \alpha)/(n-k \alpha)$, and so on, as in the general definition of a size-biased permutation, just with the usual
size $n_i$ of a cluster replaced  by $n_i - \alpha$. 
In particular, for $0 < \alpha < 1$ and $n \ge 3$ the two random compositions $\Nright_{\bullet:n}$ and $\Nstar_{\bullet:n}$ are not identically distributed. 
Remarkably, the conclusion \eqref{sbmagic:alpha} holds not only for $\GEM(\alpha,\theta)$, but also for the sampling from $\Pbul$
the size-biased presentation of frequencies in any of the Gibbs$(\alpha)$ models introduced in 
\cite[Theorem 4.6]{MR2245368}  %%% CSP
and studied further in  \cite{MR2160320}. %%AUTHOR = {Gnedin, A. and Pitman, J.}, TITLE = {Exchangeable {G}ibbs partitions and {S}tirling triangles},

We note that Theorem \ref{thm:Mn} for $\GEM(0,\theta)$ yields also the following further corollary, which identifies the known distribution
of the minimal order statistic $X_{1:n}$. This can be 
read from a result for the infinitely many alleles model, due to 
Saunders, Tavar\'e, and Watterson \cite[Theorem 8]{MR753512} %% AUTHOR = {Saunders, Ian W. and Tavar{\'e}, Simon and Watterson, G. A.}, TITLE = {On the genealogy of nested subsamples from a haploid
and the fact that this model generates age-ordered $\GEM(0,\theta)$ frequencies. See 
Donnelly and Tavar\'e \cite{MR827330} %%, AUTHOR = {Donnelly, Peter and Tavar{\'e}, Simon}, TITLE = {The ages of alleles and a coalescent},
and Donnelly \cite[Proposition 3.5]{MR865115} %%% AUTHOR = {Donnelly, Peter}, TITLE = {Partition structures, {P}\'olya urns, the {E}wens sampling
for further discussion. 
The independence assertion in the corollary is easily shown to hold in sampling from any $\RAM$ with i.i.d. factors, due to the regenerative property of these models
discussed in \cite{MR2122798}. %    AUTHOR = {Gnedin, Alexander and Pitman, Jim}, %     TITLE = {Regenerative composition structures},

\begin{corollary}
\label{indpcemink}
In sampling from $\GEM(0,\theta)$, 
the minimum of the sample, $X_{1:n} = 1 + G_{n:n}$ has a shifted geometric$(n/(n+\theta))$ distribution, and
$X_{1:n}$ is independent of the pair of random compositions $\Nright_{\bullet:n}$ and $\Nstar_{\bullet:n}$, hence also independent of the Ewens$(\theta)$ distributed partition of $n$ generated by the sample.
\end{corollary}
\begin{proof}
The independence of $G_{n:n}$ and $\Nright_{\bullet:n}$ is clear, because $\Nright_{\bullet:n}$ is a function of the $G_{i:n}$ $1 \le i \le n-1$.
But by exchangeability, conditionally given the entire collection of order statistics,  $\Nstar_{\bullet:n}$ is just a size-biased permutation of $\Nright_{\bullet:n}$.
So the order statistics and $\Nstar_{\bullet:n}$ are conditionally independent given $\Nright_{\bullet:n}$, from which the conclusion follows easily.
\end{proof}

\subsection{Combinatorial limit theorems}
\label{sec:combi}

Combinatorial models often involve exchangeable random partitions of $[n]$ into a collection of subsets of various sizes, 
typically connected components of a graph associated with the model, such as the cycles of a permutation, trees in a forest, or connected components
of a mapping digraph.
It is known 
\cite{MR2032426} %% AUTHOR = {Arratia, Richard and Barbour, A. D. and Tavar{\'e}, Simon}, TITLE = {Logarithmic combinatorial structures: a probabilistic 
\cite{MR2245368}  %%% CSP
that in many models for such
a combinatorial structure picked uniformly at random,  the sequence $\Ndec_{\bullet:n}$  of ranked component sizes converges in law after scaling by $n$: 
\begin{equation}
\label{asdeconv}
n^{-1} (\Ndec_{1:n}, \Ndec_{2:n} , \ldots )  \convd   (\Pdec_1, \Pdec_2 , \ldots ) \sim \PD(\alpha, \theta) ,\qquad n\to\infty,
\end{equation}
for some $(\alpha,\theta)$, where $\PD(\alpha,\theta)$, the {\em Poisson--Dirichlet distribution with parameters $(\alpha,\theta)$} 
is the distribution of the decreasing rearrangement of $\GEM(\alpha,\theta)$ 
\cite{MR1434129}. %% AUTHOR = {Pitman, Jim and Yor, Marc}, TITLE = {The two-parameter {P}oisson-{D}irichlet distribution derived from a stable subordinator},
According to the general theory of such limit distributions \cite{MR996613}  %%% Author = {Donelly Joice}, 
\cite{MR1659532}, % Author = Gnedin, Title = { On convergence and extensions of size-biased permutations}
this is equivalent to the corresponding convergence 
\begin{equation}
\label{asconv}
n^{-1} (\Nstar_{1:n}, \Nstar_{2:n} , \ldots )  \convd  (P^*_1, P^*_2 , \ldots )  \sim  \GEM(\alpha, \theta),\qquad n\to\infty,
\end{equation}
for the size-biased reordering $\Nstar_{\bullet:n}$ of the component sizes,
where the limit has the $\ISBP$ $\GEM(\alpha,\theta)$ distribution.
% of relative sizes of components in sized-biased random order,
%which for an exchangeable random partition can always be taken to be the order of least elements of components.
The treatment of 
\cite{MR2032426} %% AUTHOR = {Arratia, Richard and Barbour, A. D. and Tavar{\'e}, Simon}, TITLE = {Logarithmic combinatorial structures: a probabilistic 
presents a large number of such examples with $\alpha = 0$. 
An example with $\alpha = \theta = 1/2$ is provided by the tree components of a uniform random mapping digraph 
\cite[(9.7)]{MR2245368}.  %%% CSP
There are many similar examples, with ranked and size-biased compositions of $n$ derived from other constructions. For instance, if the partition of $n$
is the decreasing arrangement of lengths of excursions of an aperiodic Markov chain away from some recurrent state $0$ run for $n$ steps, and the return time of the state is in the
domain of attraction of the stable law of index $\alpha \in (0,1)$, then it is known 
\cite{MR1434129} %% AUTHOR = {Pitman, Jim and Yor, Marc}, TITLE = {The two-parameter {P}oisson-{D}irichlet distribution derived from a stable subordinator},
that \eqref{asdeconv} holds for this 
$\alpha$ with $\theta = 0$, for the Markov chain started in state $0$,  and with the same $\alpha$ with $\theta = \alpha$ for the Markovian bridge obtained by further conditioning to return to state $0$ at a late time $n$. 
In this setting,  the lengths of excursions are most naturally listed in their order of creation by the Markov chain,  
which is neither ranked nor size-biased. Still, the analysis of such limit laws for excursions is assisted by the device of deliberately size-biasing the order of excursions,
to create a $\GEM(\alpha,\theta)$ limit as in \eqref{asconv} which is much easier to deal with than the $\PD(\alpha,\theta)$ limit of ranked lengths.

In any of these settings where $\PD(\alpha,\theta)$ and $\GEM(\alpha,\theta)$ arise hand-in-hand as limit laws for ranked and size-biased counts of some kind,
it was shown by Kingman 
\cite{MR509954} %%, AUTHOR = {Kingman, J. F. C.}, TITLE = {The representation of partition structures},
that the structure of the limit theorems extends to corresponding limit theorems for sampling components of the structure of size $n$.
To set this up, consider a limited sample of $n$ elements from a much larger set of size $\Nbig$ supporting a combinatorial structure whose ranked relative component sizes 
$\Nbig^{-1}\Ndec_{\bullet:\Nbig}$ are well approximated in distribution by $\PD(\alpha,\theta)$ for some $\alpha,\theta$.
Let each component in the structure of size $\Nbig$ be labeled by its index of appearance in a size-biased listing of components. 
If the components of the combinatorial structure generate an exchangeable partition of $[\Nbig]$, 
these labels can be assigned to components in order of their least elements.
% as in the $\CRP$ construction of the partition of $[\Nbig]$.
If the components of the combinatorial structure do not generate an exchangeable partition of $[\Nbig]$, as in the case of excursion lengths of a Markov chain run for $\Nbig$ steps,
let the component labels be assigned by a size-biased random permutation of component sizes, as in Section~\ref{sec:species}.
The following proposition is an easy consequence of Kingman's theory of  partition structures 
\cite{MR509954} %%, AUTHOR = {Kingman, J. F. C.}, TITLE = {The representation of partition structures},
\cite{MR996613} %%, AUTHOR = {Donnelly, Peter and Joyce, Paul}, TITLE = {Continuity and weak convergence of ranked and size-biased permutations on the infinite simplex},
\cite{MR1659532}: %%AUTHOR = {Gnedin, Alexander V.}, TITLE = {On convergence and extensions of size-biased permutations},

\begin{proposition}
Let $\UU{i}{\Nbig}$ for $1 \le i \le n < \Nbig$ be a  simple random sample of size $n$ from $[\Nbig]$, either with or without replacement, and let $\XX{i}{\Nbig}$ be the label of the component of the combinatorial structure that 
contains $\UU{i}{\Nbig}$, for a size-biased labeling of components, that is independent of the random sample $\UU{i}{\Nbig}, 1 \le i \le n$.
Suppose there is the convergence in distribution \eqref{asconv} of size-biased relative component sizes to $\GEM(\alpha,\theta)$ with $n$ replaced by $\Nbig \to \infty$. 
Then for each fixed $n$ there is convergence of joint distributions
\[%begin{equation}
( \XX{i}{\Nbig}, 1 \le i \le n ) \convd (X_i, 1 \le i \le n) \mbox{ a sample from } \GEM(\alpha, \theta) \mbox{ as } \Nbig \to \infty,
\]%end{equation}
which implies also convergence in distribution of corresponding  order statistics, counts and gaps to those derived from the $\GEM$ sample.
\end{proposition}

This proposition shows how any exact result for a sample of size $n$ from a $\GEM$ can be turned into the conclusion of a limit theorem for sampling 
from various random combinatorial structures. Just that a double sampling process is involved, much as in Proposition \ref{prop:species}, which acquires further interpretations
in this context.
The sample of size $n$ may be regarded as an initial sample of size $n$, as in Proposition \ref{prop:species}.
Then there needs to be a secondary sample, to generate a size-biased labeling of components, run at least long enough to allocate a label to every component that intersects the initial sample. 
The limit in distribution as $\Nbig \to \infty$ of the list of secondary labels found in the primary sample of size $n$ is then a size $n$ sample from $\GEM(\alpha,\theta)$.
The case of exchangeable partitions is particularly natural, as the initial sample of size $n$ can be taken to be the set $[n]$ instead of a random subset of size $n$, and the secondary labeling of components 
can be taken to be the order of least elements of components, starting the labeling afresh after the initial sample of size $n$, as in Proposition \ref{prop:species}. 
As $\Nbig\to \infty$ there is a negligible difference between this construction and a completely independent size-biased listing of components,  so the
conclusions of the above Proposition are valid in either setup.

For application to the excursions of a Markov chain, given the path of the Markov chain of length $\Nbig$, due to lack of exchangeability, two random samples are required, 
one to choose $n$ sample times from $[\Nbig]$, and the other to assign secondary labels to excursions in their order of discovery by a random permutation of $[\Nbig]$. As in the
exchangeable case, it makes no difference if the second random sample is just a continuation of the first, avoiding the first $n$ elements drawn and continuing without replacement until all $\Nbig$
time points have been covered, and all the excursions found. In this scenario there a tiny probability that the first $n$ time points sampled might find excursions which were not part of the subsequent sample,
and hence to which no label can be assigned, but the probability of this event and all other differences in the distribution of the first $n$ sample labels is negligible in the large $\Nbig$ limit, 
because the assumption of convergence to $\GEM(\alpha,\theta)$ proper frequencies means that with overwhelming probability these first $n$ sample points will fall in some collection of $K_n$ excursions each of
which acquires a significant fraction of the rest of the sample points in $[\Nbig]$ and $\Nbig \to \infty$.

The interpretation in this setting of large $n$ limit theorems for sampling a $\GEM$ distribution is not immediate.
But such interpretations might still be made with adequate provision of a limit regime with both $n$ and $\Nbig$ tending to infinity. 
To adequately justify such a double limit theorem, some estimate of the adequacy of the approximation of $\Nbig^{-1}\Nstar_{\bullet: \Nbig}$ by $\GEM(\alpha,\theta)$ 
would be required, such as that provided in \cite{MR2195574} %%%AUTHOR = {Arratia, Richard and Barbour, A. D. and Tavar{\'e}, Simon}, TITLE = {A tale of three couplings: {P}oisson-{D}irichlet and {GEM}
for the random permutation statistics approaching $\GEM(0,1)$.

\section{The maximum of a sample from a random discrete distribution}
\label{sec:max}

This section develops some general formulas for the distribution of the maximum of a sample from a
random discrete distribution on positive integers. These formulas allow us to check 
Corollary \ref{crl:Max} 
without appeal to Ignatov's representation
of $\GEM(0,\theta)$. Our interest in this approach is that it at least gives us an explicit if difficult formula for the distribution of the
maximum of a sample from $\GEM(\alpha,\theta)$.

We begin with a well known representation of the probability generating function of a discrete random variable in terms of its tail probabilities.

\begin{lemma}[{\cite[p.~265, Theorem 1]{MR0228020}}]%%%%, AUTHOR = {Feller, William}, TITLE = {An introduction to probability theory and its applications.  {V}ol. {I}},
\label{lem:Fel}
For $X$ a non-negative integer valued random variable, the
probability generating function 
\[
\BE z^X := \sum_{n = 0}^\infty \BP[ X = n  ]z^{n} 
\]
may be represented for $|z| < 1$ as
\begin{align}
\label{feller1} \BE z^{X} &{}=   1 - (1-z) \sum_{m=0}^\infty \BP[ X > m) z^m \\
\label{feller2} &{}=  (1-z) \sum_{m=0}^\infty \BP[ X \le m) z^m .
\end{align}
\end{lemma}

This allows us to provide the following general expression for the distribution of the maximum of a sample from a random discrete distribution:

\begin{lemma}\label{lem:Mn}
Let $M_n = \max_{1 \le k \le n} X_k$ for a sequence of exchangeable positive integer valued random variables $X_1,\ldots,X_n$ which are conditionally
i.i.d.\  $\Pbul$ given some random discrete distribution $\Pbul$ with $R_k:= 1 - \sum_{j=1}^k P_j \downarrow 0$ a.s. 
%%%$\mathbf{Y} = (\YY_1, \YY_2, \ldots)$ with $0 \le \YY_1 \le \YY_2 \le \cdots \uparrow 1$ a.s., with
%%%$\BP[X_1\le k|\mathbf{Y}]=\YY_k$ for $k = 1,2, \ldots$.  
Then the probability generating function of $M_n-1$ admits the representation
\begin{equation}
\label{eq:genfunc}
\BE z^{M_n - 1} = (1-z) \sum_{j=0}^n \binom{n}{j } (-1)^j 
\sum_{k=1}^\infty \BE  R_k^j \, z^{k-1}.
\end{equation}
\end{lemma} 

\begin{proof} 
We apply \eqref{feller2} to $X = M_n - 1$.
For $k = 1,2, \ldots$ the term for $m = k-1$ is evaluated by taking expectations in the following identity:
\[%begin{equation}
\BP[M_n-1 \le k-1|\Pbul] =\BP[M_n\le k|\Pbul] =(1- R_k)^n =\sum_{j=0}^n\binom{n}{j}(-1)^j R_k^j  .
\]%end{equation}
Now \eqref{eq:genfunc} follows easily from \eqref{feller2}.
\vskip-1.5ex
\end{proof}

Since the $\GEM(\alpha,\theta)$ model makes 
the $ 1 - H_i$ independent with beta$ (\theta + i \alpha, 1 - \alpha)$ distributions, for $j = 0,1, \ldots$
\[%begin{equation}\label{tailmoms}
\BE _{\alpha,\theta}( 1 - H_i)^j = \frac{B(\theta+i\alpha+j,1-\alpha)}{B(\theta+i\alpha,1-\alpha)}
= \frac{ ( \theta + i \alpha )_j }{ (\theta  + (i-1) \alpha  + 1)_j }
\]%end{equation}
hence the $\GEM(\alpha,\theta)$ tail moment formula
\begin{equation}
\label{pdtails}
\BE _{\alpha,\theta} R_k^j  = \prod_{i= 1}^k \frac{ ( \theta + i \alpha )_j }{ (\theta  + (i-1) \alpha  + 1)_j }  .
\end{equation}
Thus we obtain:

\begin{proposition}
The probability generating function of $M_n-1$ for the maximum $M_n$ of a sample of size $n$ from $\GEM(\alpha,\theta)$ is given by formula
\eqref{eq:genfunc} for the $\GEM(\alpha,\theta)$ tail moments  \eqref{pdtails}.
\end{proposition}

For $j=1$ the product in \eqref{pdtails} gives the tail probability formula
\begin{equation}
\label{tailprob}
\BP _{\alpha,\theta}[X_1 > k ] = \prod_{i= 1}^k \frac{  \theta + i \alpha }{  1 + \theta  + (i-1) \alpha } 
\end{equation}
for $X_1$ a sample of size $1$ from $\GEM(\alpha,\theta)$. For $\alpha = 0$ the product reduces to $(\theta/(1 + \theta))^k$.
This geometric distribution of $X_1$ with parameter $1/(1 + \theta)$ was indicated by 
%Donnelly and Tavar\'e \cite{xxx} 
Engen~\cite{MR0411097}
in the context of %genetics 
ecological models.  Formula 
\eqref{eq:genfunc}
in this case reduces easily to the familiar formula for the probability generating function of the geometric distribution of $X_1 - 1$ on non-negative integers, 
\[%begin{equation}
\BE _{0,\theta} z^{X_1 - 1} = \frac{ 1 }{ 1 - \theta(z-1) } = 1 + \theta (z-1) + \theta^2 (z-1)^2 + \cdots
\]%end{equation}
whose binomial moments
can be read from the expansion in powers of $(z-1)$:
\begin{equation}
\label{geomoms}
\BE _{0,\theta} {X_1 - 1 \choose k } = \theta^k   \qquad (k = 0,1, \ldots ).
\end{equation}
For $0 < \alpha < 1$ the tail probabilities  
\eqref{tailprob} may be recognized as the terms of a hypergeometric series. This allows the
following evaluation in terms of the Gaussian hypergeometric function ${}_2F_1$:
%\begin{equation}
%\label{X1genfunc}
\[
\BE _{\alpha,\theta} z^{X_1 - 1} = 1 - \frac{(\alpha + \theta)}{ (1 + \theta)}  \, (1-z) \, 
\setlength\arraycolsep{1pt}
{}_2 F_1\left(\begin{matrix} 1, \, 2 + \theta/\alpha  \\
1 + (1+\theta)/\alpha\end{matrix};z\right)
\]%\end{equation}
with associated binomial moments 
\begin{equation}
\label{hypermoms}
\BE _{\alpha,\theta} {X_1 - 1 \choose k } = \prod_{i=1}^k \frac{ \theta + i \alpha }{ 1 - (i+1) \alpha } \qquad \mbox{ if } 0 \le  \alpha < \frac{1}{ k+1}
\end{equation}
and $\infty$ otherwise.  Note that \eqref{hypermoms} reduces correctly to \eqref{geomoms} for $\alpha = 0$.
The case $k=1$ of \eqref{hypermoms} is due to Kingman \cite[(18)]{MR0368264} %% AUTHOR = {Kingman, J. F. C. }, TITLE = {Random discrete distributions}, 
for $\alpha = 0$ and \cite[(58)]{MR0368264} for $\theta = 0$. %% AUTHOR = {Kingman, J. F. C. }, TITLE = {Random discrete distributions}, 
The case of \eqref{hypermoms} for $\theta = 0$ and general $k = 1,2, \ldots$ was recently derived by 
Leisen, Lijoi, and Paroissin \cite[Proposition 1]{MR2845896} %%% TITLE = {Limiting behavior of the search cost distribution for the 
using a much more difficult approach.

For $j  > 1$ only in the case $\alpha = 0$ does there seem to be much simplification in \eqref{eq:genfunc}. Then we can proceed as follows:

\begin{proof}[Computational proof of Corollary~\ref{crl:Max}]
For $\alpha = 0$ in \eqref{pdtails} we find that
\[
\BE _{0,\theta} R_k^j  = \left( \frac{\theta } { \theta + j } \right)^k
\]
and the series in \eqref{eq:genfunc} becomes
\[
\sum_{k=1}^\infty \BE  _{0,\theta} R_k^j z^{k-1} = z^{-1} \sum_{k=1}^\infty  \left( \frac{\theta  z } { \theta + j } \right)^k =  \frac{ \theta }{ j + \theta(1-z)}
\]
hence
\begin{equation}\label{eq:genfM}
\BE_{0,\theta} z^{M_n - 1 } = ( 1 -z ) \sum_{j=0}^n \binom{n}{ j } \frac{ (-1)^j \theta }{ j + \theta (1 - z ) } = \prod_{i= 1}^n \frac{ i } { i + \theta ( 1 - z) }\,.
\end{equation}
The last equality is the well-known partial fraction decomposition (see, e.g.~\cite[Eq.~5.41]{MR1001562})
\begin{equation}\label{eq:partfr}
\frac{1}{(x)_{n+1}}=\frac{1}{n!}\sum_{j=0}^n\binom{n}{j}\frac{(-1)^j}{x+j}
\end{equation}
which can be verified, for instance, by multiplying \eqref{eq:partfr} by $x+k$ and plugging in $x=-k$ for $k=0,1,\ldots,n$.
Since the factors in the right-hand side of \eqref{eq:genfM} are the probability generating functions of geometric variables $G_i$ with parameters $i/(i+\theta)$,
the conclusion of Corollary~\ref{crl:Max} follows.
\end{proof}

%\medskip

\begin{remark}
Looking on the form of \eqref{eq:Mnassum} it is tempting to suppose that $G_n$ is the difference $M_n-M_{n-1}$ and is independent of $M_{n-1}$. However this is
not the case, because the new sample $U_{n+1}$ gets into an arbitrary position~$\ell$ in the order statistics and hence changes 
a value of $G_{n-\ell}$.  Moreover, unlike the independent case, the successive maxima do not form a Markov chain. Heuristically, this happens because 
knowledge of the history provides some information about the realization of $\mathbf{Y}$. It can be shown, for instance, that
$\BP\bigl[M_1=j,M_2=M_3=\ell|M_1=j,M_2=\ell\bigr]$ for $j<\ell$ depends on $j$, but we omit this calculation.
A similar issue arises in the identity in distribution \eqref{expid} relating the distribution of the maximum $M_n$ of $n$ i.i.d.\  exponential variables to the sum 
$T_n$ of scaled exponentials.  But that identity fails to hold jointly as $n$ varies for a more obvious reason: $\BP(M_n = M_{n-1}) > 0$ while $\BP(T_n = T_{n-1}) = 0$.
\end{remark}

\section{A generalization in the $\GEM(0,\theta)$ case}\label{sec:Ntheta}
The result of Corollary \ref{crl:Max} can be generalized as follows.
According to \eqref{xue}, the $\GEM(0,\theta)$ model makes $M_n - 1 = N_F (0, U_{n:n}]$ a sum of independent geometrics, 
where $N_F:= (N_F(0,u], 0 \le u < 1)$ is the $\GEM(0,\theta)$ barrier process, which is Poisson  with intensity $\theta (1-u)^{-1} du$ at $u \in (0,1)$,
and $U_{n:n}$ is independent of $N_F$.
Instead of $N_F (0,U_{n:n}]$, consider $N_F (0,\beta]$ for $\beta$ with a suitable beta distribution, independent of $N_F$.

\begin{theorem}\label{thm:Gsum2}
For $n\in\BN$ and $\theta,b>0$, let $\beta_{n,b}$ with the\/ {\upshape beta}$(n,b)$ density at $u$ proportional to $u^{n-1} (1-u)^{b-1}$
be independent of the $\GEM(0,\theta)$ barrier process $N_F$.
Then 
\begin{equation}\label{eq:Gsum2}
N_F (0,\beta_{n,b}]\deq \sum_{i= 1}^n G_i (b,\theta)
\end{equation}
where the $G_i(b,\theta)$ are independent with geometric$(\ptau_i(b,\theta))$ distributions, for
\begin{equation}\label{eq:taubtheta}
\ptau_i(b,\theta) := \frac{ b + i - 1 }{ b + i - 1 + \theta } .
\end{equation}
\end{theorem}

\begin{proof}
Consider first $N_F (0,W]$ where $W$ is a random variable with some arbitrary distribution on $[0,1]$, independent of $N_F$.
For  $W = u$ fixed, the distribution of $N_F (0,u]$ is Poisson$ ( - \theta \log ( 1 - u ) )$ with the probability generating function 
\[
\BE z^{N_F(0,u] } =   \exp \left[ - (1 - z ) ( - \theta \log ( 1- u ) \right] = (1 - u ) ^{\theta (1-z)}\,.
\]
For general $W$ the distribution of $N_F (0,W]$ ranges over all mixed Poisson distributions. 
Explicitly, the probability generating function of $W$ is
\[
\BE z^{N_F(0,W] } =  \BE(1 - W ) ^{\theta (1-z)}.
\]
In particular, if $W = \beta_{a,b}$ has the beta$(a,b)$ distribution then
\begin{align*}
\BE z^{N_F(0,\beta_{a,b}] } &{}=  \BE(1 - \beta_{a,b} ) ^{\theta (1-z)} = \BE\beta_{b,a} ^{\theta (1-z)} \\
&{}= \frac{ \Gamma( b + \theta(1-z) )  } { \Gamma( a + b + \theta(1-z) ) } \frac{ \Gamma( a + b ) }{\Gamma(b) }.
\end{align*}
If $a = n$ is a positive integer, then $\frac{ \Gamma(n + b)}{ \Gamma(b) } = (b)_n:= \prod_{i = 1}^n (b + i - 1)$
so
\[%begin{equation}
\BE z^{ N_F(0,\beta_{n,b}] } = \prod_{i=1}^{n} \frac{ b + i - 1 } { b + i - 1 + \theta(1-z) } =  \prod_{i=1}^{n} \frac{ \ptau_i(b,\theta)   } {  1 - (1- \ptau_i (b,\theta)) z }
\]%end{equation}
for $\ptau_i(n,\theta)$ as in~\eqref{eq:taubtheta}. Since the $i$-th factor is the probability generating function for 
the geomet\-ric($\ptau_i(n,\theta)$) distribution, the claim~\eqref{eq:Gsum2} follows.
\end{proof}

\begin{remark}
Notice that $U_{n,n} \deq \beta_{n,1}$,  
so \eqref{eq:Gsum2} is a generalization of \eqref{eq:Mnassum}.
\end{remark}

\section{The maximum of a $\GEM(\alpha,\theta)$ sample for $0<\alpha<1$}
\label{sec:max2}

The technique of the previous sections does not seem to work for the case $0<\alpha<1$. However the asymptotics of the $\GEM$
distribution in this case are known sufficiently well to find the asymptotic behavior of $M_n$ as $n\to\infty$. In particular,
it is known that $\GEM(\alpha,\theta)$ frequencies $P_i$ almost surely decay as random factors of $i^{-1/\alpha}$. Similar 
behavior is also known for the sampling from a random branching process model introduced by Robert and Simatos~\cite{MR2541191}
where different but similar aspects of samples are studied, such as the limit behavior of the first unoccupied box, as the sample size grows.

A key role in the study of the $\GEM(\alpha,\theta)$ distribution for the case $0<\alpha<1$ is played by the notion
of the $\alpha$-diversity of the exchangeable sample. 
%Denote by $K_n$ the number of distinct values in 
%a sample of size $n$. 
It is known \cite[Th.~3.8]{MR2245368} that for $K_n$ defined by \eqref{kndef} from a $\GEM(\alpha,\theta)$  sample
%    AUTHOR = {Pitman, J.},
%     TITLE = {Combinatorial stochastic processes}, 
with $0<\alpha<1$ and $\theta>-\alpha$ there exists a limit
\[
\lim_{n\to\infty}\frac{K_n}{n^\alpha}=D_{\alpha}>0 \mbox{ almost surely } (\BP_{\alpha, \theta} )
\]
and also in $p$\/-th mean for every $p>0$. The distribution of the limiting random variable 
$D_{\alpha}$, 
which depends on $\theta$,  is known as the \textit{$\alpha$-diversity} and is
determined by its moments
\begin{equation}\label{eq:diver}
\BE_{\alpha,\theta} D_\alpha ^p=\frac{\Gamma(\theta+1)}{\Gamma(\tfrac{\theta}{\alpha}+1)}\,\frac{\Gamma(p+\tfrac{\theta}{\alpha}+1)}{\Gamma(p\alpha+\theta+1)}\,.
\end{equation}
Moreover, the $\alpha$-diversity $D_{\alpha}$ is  a.s.\ determined by $\Pbul$ and 
\[
\BP_{\alpha,\theta} [X_1 > k|\Pbul]\sim \alpha D_{\alpha,\theta}^{1/\alpha}\, k^{1-1/\alpha}
\mbox{ almost surely } (\BP_{\alpha, \theta} )\mbox{ as } k\to\infty,
\]
see \cite[Sec.~10]{MR2318403} or \cite[Lemma~3.11]{MR2245368}.
For such a power law it is well known that the maximum of an independent sample of size $n$ converges in distribution
to the Fr\'echet distribution. Namely, writing for short $\gamma=1/\alpha-1$, for any fixed $x>0$
\begin{align*}
\BP_{\alpha,\theta} \bigl[M_n\le xn^{1/\gamma}\bigm|\Pbul \bigr]&{}=\bigl(1-\BP_{\alpha,\theta}\bigl[X_1 > xn^{1/\gamma}\bigm|\Pbul \bigr]\bigr)^n\\
&{}\sim \left(1-\alpha D_{\alpha,\theta}^{1/\alpha}\frac{x^{-\gamma}}{n}\right)^n\mbox{ almost surely } (\BP_{\alpha, \theta} )\\
&{}\to \exp\bigl(-\alpha D_{\alpha,\theta}^{1/\alpha}x^{-\gamma}\bigr),\qquad n\to\infty.
\end{align*}
Hence, by integration with respect to the distribution of the $\alpha$-diversity, we have the following result.

\begin{theorem}\label{thm:Mnalpha}
Let $M_n$ be the maximum of a size $n$ $\GEM(\alpha,\theta)$ exchangeable sample with $0<\alpha<1$ and $\theta>-\alpha$. Then
for each $x>0$
\begin{equation}\label{eq:Mndistr}
\BP_{\alpha,\theta} \bigl[M_n\le xn^{\alpha/(1-\alpha)}\bigr]\to \BE_{\alpha,\theta} \exp\bigl(-\alpha D_{\alpha}^{1/\alpha}x^{-(1-\alpha)/\alpha}\bigr) \mbox{ as }  n\to\infty.
\end{equation}
%where $D_{\alpha,\theta}$ is the random variable with the distribution determined by its moments \eqref{eq:diver}.
\end{theorem}

\begin{remark}
For the case $\alpha=0$ the asymptotic result $K_n\sim M_n\sim \theta\log n$ almost surely  $(\BP_{0, \theta} )$ 
of \cite{MR2538083} %%% AUTHOR = {Gnedin, Alexander V. and Iksanov, Alexander M. and Negadajlov, Pavlo and R{\"o}sler, Uwe}
shows that asymptotically $K_n$ and $M_n$ have the same behavior. For $\alpha>0$ the situation
is different: $K_n$ should be divided by $n^{\alpha}$ to get a proper limit, and $M_n$ grows much faster as a random factor of $n^{\alpha/(1-\alpha)}$.
\end{remark}

Note that \eqref{eq:Mndistr} expresses the cumulative distribution function of $\lim n^{-\alpha/(1-\alpha)}M_n$ 
evaluated at $x$ as the 
Laplace transform $\BE_{\alpha,\theta} \bigl[\re^{-y D_{\alpha}^{1/\alpha}}\bigr]$ evaluated at $y=\alpha x^{-(1-\alpha)/\alpha}$.
Since the $\BP_{\alpha,\theta}$ moments of $D_{\alpha}$ given by \eqref{eq:diver} determine its distribution, we can obtain an explicit but clumsy
expression for the limiting distribution function \eqref{eq:Mndistr}. 

\begin{theorem}\label{prop:Mnalphaasint}
For the $\BP_{\alpha,\theta}$ distribution of $D_{\alpha}$ determined by the moment function \eqref{eq:diver},
\begin{align}\label{eq:Mndistrasint}
%\BP\bigl[M_n\le xn^{\alpha/(1-\alpha)}\bigr]\to 
\notag
\BE_{\alpha,\theta} \exp&{}\bigl(-\alpha D_{\alpha,\theta}^{1/\alpha}x^{-(1-\alpha)/\alpha}\bigr)\\
&\!\!{}=\frac{2\alpha^{1-\theta-\alpha}\, \Gamma(\theta+1)}{\Gamma(\tfrac{\theta}{\alpha}+1)}
x^{(1-\alpha)(\theta/\alpha+1)}\int_0^\infty v^{\theta+2\alpha-1}\re^{-(v^2/\alpha)^\alpha x^{1-\alpha}} J_\theta(2v)\,dv,
\end{align}
where $J_\theta$ is the Bessel function.
\end{theorem}

\begin{proof}
Writing for short $y=\alpha x^{-(1-\alpha)/\alpha}$ we have, for any $c>0$, 
\[
\re^{-yD_{\alpha}^{1/\alpha}}=\frac{1}{2\pi\ii}\int_{c-\ii\infty}^{c+\ii\infty}\Gamma(s)\bigl(yD_{\alpha}^{1/\alpha}\bigr)^{-s}ds
\]
because $\re^{-y}$ and $\Gamma(s)$ form the Mellin pair. We refer to \cite{MR0352890}
%    AUTHOR = {Oberhettinger, Fritz},
%     TITLE = {Tables of {M}ellin transforms},
for the necessary information about Mellin's transform. By analyticity the expression~\eqref{eq:diver} for $\BP_{\alpha,\theta}$ moments of $D_{\alpha}$ is valid
also for complex $p$ at least with $\Re p>-1-\tfrac{\theta}{\alpha}$.  
Hence taking expectation and applying Fubini's theorem yields
\begin{equation}\label{eq:LaplaceD}
\BE_{\alpha,\theta}\bigl[\re^{-yD_{\alpha,\theta}^{1/\alpha}}\bigr]=\frac{1}{2\pi\ii}\frac{\Gamma(\theta+1)}{\Gamma(\tfrac{\theta}{\alpha}+1)}
\int_{c-\ii\infty}^{c+\ii\infty}\Gamma(s)\,\frac{\Gamma(\tfrac{\theta-s}{\alpha}+1)}{\Gamma(\theta-s+1)}y^{-s}ds
\end{equation}
for $0<c<\alpha+\theta$.
Now, $\Gamma(s)/\Gamma(\theta-s+1)$ is the Mellin transform of $y^{-\theta/2}J_\theta(2\sqrt{y})$ in the fundamental strip $0<\Re s<\tfrac{\theta}{2}+\tfrac{3}{4}$ (\cite[II.5.38]{MR0352890}, where there is a misprint in the right bound) and $\Gamma(\tfrac{\theta-s}{\alpha}+1)$ is the 
Mellin transform of $\alpha y^{-\alpha-\theta}\re^{-y^{-\alpha}}$ for $\Re s<\alpha+\theta$, by the standard transformations
of the Mellin pair $\re^{-y}$ and $\Gamma(s)$. Hence 
 their product in the intersection of fundamental strips is the Mellin transform of the multiplicative convolution and for $0<c<\min\{\tfrac{\theta}{2}+\tfrac{3}{4},\alpha+\theta\}$
by the inversion formula 
\[
\frac{1}{2\pi\ii}
\int_{c-\ii\infty}^{c+\ii\infty}\Gamma(s)\,\frac{\Gamma(\tfrac{\theta-s}{\alpha}+1)}{\Gamma(\theta-s+1)}y^{-s}ds
=\alpha\int_{0}^\infty  (y/u)^{-\alpha-\theta}\re^{-(y/u)^{-\alpha}}u^{-\theta/2}J_\theta(2\sqrt{u})\,\frac{du}{u}\,.
\]
Plugging this into \eqref{eq:LaplaceD}, changing the variable $v=\sqrt{u}$ and 
returning to the variable $x$ yields the result.
\end{proof}

The right-hand side of~\eqref{eq:Mndistrasint} does not seem to allow much simplification for general $\alpha$. For some rational
$\alpha$ Mathematica evaluates this integral in terms of the hypergeometric function. However for $\alpha=1/2$ the integral can be taken explicitly
and leads to a simple expression.  In this case the integral is the Mellin transform of 
$f(v)=\re^{-v\sqrt{2 x}} J_\theta(2v)$ evaluated at $\theta+1$. According to \cite[I.10.7]{MR0352890}
\[
\int_0^\infty v^{s-1}\re^{-v\sqrt{2 x}} J_\theta(2v)\,dv=(2x)^{-(s+\theta)/2}\frac{\Gamma(\theta+s)}{\Gamma(\theta+1)}\,
\setlength\arraycolsep{1pt}
{}_2F_1\left(\begin{matrix}
\tfrac{\theta+s}{2},\,\tfrac{\theta+s+1}{2}\\ \theta+1\end{matrix};-\tfrac{2}{x}\right)\,
\]
and the last expression simplifies for $s=\theta+1$ because
\[
\setlength\arraycolsep{1pt}
{}_2F_1\left(\begin{matrix}a,\,b\\b\end{matrix};z\right)=\sum_{k=0}^{\infty}\frac{(a)_k}{k!}z^k=\frac{1}{(1-z)^a}
\]
for $|z|<1$ and by analyticity also for all $z$ with $\Re z<1$. Hence
\[
\int_0^\infty v^{\theta}\re^{-v\sqrt{2 x}} J_\theta(2v)\,dv=\frac{\Gamma(2\theta+1)}{\Gamma(\theta+1)}\,
\frac{1}{(2x+4)^{\theta+1/2}}
\]
and plugging it into \eqref{eq:Mndistrasint} gives the following result.

\begin{corollary}
Let $M_n$ be the maximum of a size $n$ $\GEM(\tfrac{1}{2},\theta)$ exchangeable sample with  $\theta>-\tfrac12$. Then
$M_n/n$ converges in distribution as $n\to\infty$ to a random variable with the cumulative distribution function $\bigl(x/(x+2)\bigr)^{\theta+1/2}$.
\end{corollary}

Some simplification is also possible for rational $\alpha\ne 1/2$ using the representation of the
$\alpha$-diversity in terms of product of random variables with beta and gamma distributions given in \cite[Sec.~8]{MR2676940}.
%    AUTHOR = {James, Lancelot F.},
%     TITLE = {Lamperti-type laws},}.

\section{Limit laws for the number of missing values and number of ties at the maximum}
\label{sec:tiesatmax}

This section offers some complements to the analysis of limit laws for $M_n$ in $\GEM(0,\theta)$ settings,
following the work of Gnedin et al.~\cite{MR2538083}.  %%% AUTHOR = {Gnedin, Alexander V. and Iksanov, Alexander M. and Negadajlov, Pavlo and R{\"o}sler, Uwe}

It was observed in other notation in \cite[(19)]{MR2538083} that in sampling from $\GEM(0,\theta)$,
for the number of missing values in the range $K_{0:n} := M_n - K_n $ there is the representation in distribution
\eqref{deltan} in terms of independent geometric$(p_i)$ random variables $G_{i:n}$ with $p_i:=i/(\theta+i)$. Writing now
$G(p_i)$ instead of $G_{i:n}$ to emphasize the lack of dependence on $n$ in this representation, apart from the trivial term $G_{n:n}$ which
obviously converges almost surely to $0$ as $n \to \infty$, we deduce easily that
\begin{equation}
\label{kzerolim}
K_{0:n} \convd K_{0:\infty}:=   \sum_{i = 1}^\infty ( G(p_i) - 1 ) _+
\end{equation}
This is just an explicit presentation of a random variable with the limit distribution of 
$K_{0:n}$ as $n \to \infty$ which was described in \cite[Proposition 5.1]{MR2538083} by the probability
generating function
\begin{equation}
\label{delgf}
g_\theta(z):=\BE z^{K_{0:\infty}} =  \frac{ \Gamma(1 + \theta) \Gamma(1 + (1-z)\theta  ) } {\Gamma( 1 + ( 2 - z) \theta)}   \qquad( |z| \le 1)
\end{equation}
which can be read from \eqref{kzerolim} as an infinite product of modified geometric generating functions.
This product simplifies to \eqref{delgf} due to the Weierstrass product formula for the gamma function \cite[Eq.\ (1.1.3), p.~1]{MR0058756}.
%    AUTHOR = {Erd{\'e}lyi, Arthur and Magnus, Wilhelm and Oberhettinger,
%              Fritz and Tricomi, Francesco G.},
%     TITLE = {Higher transcendental functions. {V}ols. {I}, {II}},
As observed in \cite[Proposition 5.1]{MR2538083}, this distribution of $K_{0:\infty}$ is a mixed Poisson distribution with
random parameter distributed as $\theta |\log \beta_{1,\theta}|$ for $\beta_{1,\theta}$ with the beta$(1,\theta)$ distribution of
$P_1$, the first $\GEM(1,\theta)$ frequency. 

That result may by understood as a refinement of \eqref{kzerolim} in which each term $(G(p_i) - 1)_+$ is replaced by the
distributionally equivalent random variable $N_i( \theta B_{1 - p_i} \eps_i/i )$ where the $\eps_i$ are independent standard exponential 
variables, the  $B_{1 - p_i}$ are independent Bernoulli variables with the indicated parameters, independent also of the $\eps_i$, and
the $N_i$ are independent rate one Poisson processes independent of both the $\eps_i$ and the $B_{1 - p_i}$. Then there is the identity in distribution
\begin{equation}
\label{betarep}
\sum_{i=1}^\infty B_{1 - p_i} \frac{ \eps_i }{i} \ed - \log \beta_{1,\theta}\,  \mbox{ where } 1 - p_i = \theta/(i + \theta)
\end{equation}
which can be checked by computing the Laplace transform of both sides at $\lambda > 0$.
%\[%begin{equation}
%\prod_{i = 1}^\infty \left( \frac{i }{ i + \theta} + \frac{\theta}{(i + \theta)} \frac{1 }{ ( 1 + \lambda/i )}  \right) = \frac{ (1)_\lambda }{ (1 + \theta)_\lambda } = \BE \beta_{1,\theta}^\lambda
%\]%end{equation}
%using the Pochhammer symbol $(x)_\lambda:= \Gamma(x + \lambda)/\Gamma(x)$. 
This identity \eqref{betarep} is the instance $a= 1, b = \theta$ of the identity
\eqref{betaab} presented in the following proposition, which is the simpler variant for
$\log$ beta variables of a representation of $\log$ gamma variables due to Gordon 
\cite{MR1300491}. %%%, AUTHOR = {Gordon, Louis}, TITLE = {A stochastic approach to the gamma function},

\begin{proposition}
For each $a , b > 0$, for $0 < \beta_{a,b} < 1$ with  density proportional to $u^{a-1}(1-u)^{b-1}$ at $0 < u < 1$ there is the identity in distribution
\begin{equation}
\label{betaab}
\sum_{j=0}^\infty B_{b/(a+b+j)}\frac { \eps_{j} }{ a + j } \ed - \log  \beta_{a,b} ,
\end{equation}
where the $\eps_j$ are i.i.d.\ standard exponential variables, independent of a sequence of independent Bernoulli variables $B_{p_j}$ 
with parameters $p_j= b/( a + b + j )$.
\end{proposition}
\begin{proof} The well known identity in distribution $\beta_{a,b} \gamma_{a + b} \ed \gamma_a$ for independent beta and gamma
variables with the indicated parameters, and known representations of $\log $ gamma variables, show that the distribution of the non-negative random variable 
$- \log \beta_{a,b}$ is infinitely divisible with L\'evy density at $x >0$ which is 
given by the formula
\cite[p. 769]{MR1981512} %%AUTHOR = {Bose, Arup and Dasgupta, Anirban and Rubin, Herman}, TITLE = {A contemporary review and bibliography of infinitely divisible
\begin{equation}
\label{levybeta}
\frac{ x ^{-1} } { 1 - e^{-x } } ( e^{- a x} - e^{ - (a + b) x } ) = \sum_{j=0}^\infty x^{-1} ( e ^{- ( a + j) x } - e^{- (a + b + j ) x } ).
\end{equation}
But it is also well known and easily checked that the $j$th term on the right side of  \eqref{levybeta} is the L\'evy density of the infinitely divisible law of the $j$th term
in \eqref{betaab}.
So the conclusion follows easily from the additivity of L\'evy measures.
\end{proof}

For $z=0$ the generating function \eqref{delgf} gives the limiting probability of what
is called in \cite{MR2137565} %%%  AUTHOR = {Hitczenko, Pawe\l  and Knopfmacher, Arnold},  TITLE = {Gap-free compositions and gap-free samples of geometric random variables},} 
the event of a {\em complete sample} with no gaps:
\begin{equation}
\label{kmeq}
\lim_{n \to \infty} \P(K_{0:n} = 0) =  \lim_{n \to \infty} \P(K_n = M_n) = g_\theta(0)=\frac{ \Gamma(1 + \theta) ^2 }{ \Gamma( 1 + 2 \theta  ) }.
\end{equation}
If $\theta = m$ say is a positive integer, these formulas simplify by the gamma recursion $\Gamma(1+x) = x \Gamma(x)$. 
The generating function \eqref{delgf} reduces to rational function of $z$, with $m$ linear factors in the denominator. This implies
that $K_{0:\infty}$ is distributed as the sum of just $m$ independent geometrics $G_i$, with the by now familiar parameters $i/(i+\theta)$ for $1 \le i \le m$.
This yields the remarkable chain of identities in law
\[%begin{equation}\label{delmeq}
K_{0:\infty} \ed \sum_{i=1}^m G_i \ed M_m - 1  \ed X_{n:n} - X_{n-m:n} \mbox{ for all } n \ge m \mbox{ if } \theta = m  \in \BN .
\]%end{equation}
where the first $\ed$  holds only if $\theta = m  \in \BN$, but the next two $\ed$ hold for all $\theta >0$ by Theorem \ref{thm:Mn}.
Also for $\theta = m \in \BN$, the right side of \eqref{kmeq} is the inverse of the central binomial coefficient ${2 m \choose m}^{-1} \sim 2^{-2m} \sqrt{\pi m}$ as $m \to \infty$,
and the approach of this probability to $0$ is similarly rapid for $\theta \to \infty$ through real values, due to Stirling's approximation  to the gamma function.
In particular, for 
$\theta = 1$, the probability of a complete sample is simply $1/2$. This is also known \cite{MR2137565} %%%  AUTHOR = {Hitczenko, Pawe\l  and Knopfmacher, Arnold},  TITLE = {Gap-free compositions and gap-free samples of geometric random variables},} 
to be
the common value of $\P(K_n = M_n)$ for every $n$ in the  case of i.i.d.\  sampling from geometric$(1/2)$.
Some extensions of these results to more general $\RAM$s with i.i.d.\ factors will be given in \cite{ramgaps}.
%Corollary \ref{xxx} in Section \ref{xxx} extends these limit distributions from $\GEM(0,\theta)$ to more general $\RAM$s with i.i.d.\  factors.

Also either from \eqref{kzerolim} or from the probability generating function $g_\theta$ given in \eqref{delgf} 
it is easy to find the generating function for the L\'evy measure of $K_{0:\infty}$ which has atoms $\lambda_k$ in $k=1,2,\ldots$:
\[
\sum_{k=1}^\infty \lambda_kz^k=-\log g_\theta(0)+\log g_\theta(z)=\log \frac{\Gamma(1+2\theta)}{\Gamma(1+\theta)}\,\frac{\Gamma(1+(1-z)\theta)}{\Gamma(1+(2-z)\theta)}\,.
\]
Its Taylor's expansion gives an expression for the atoms $\lambda_k$ in terms of the polygamma function.

%\section{Number of maximal values}\label{sec:nummax}
\smallskip 

The right tail probability function of $L_\infty$, the limit in distribution of $L_n$  can be read from its definition \eqref{nties}:
For $k = 0, 1,2, \ldots $
%\begin{equation}
%\label{linf}
\[
\BP (  L_\infty > k ) = \BP( G_i = 0, 1 \le i \le k ) = \frac{(1)_k}{(1+ \theta)_k} \sim \frac{ \Gamma(1 + \theta) }{k^\theta} \mbox{ as } k \to \infty .
\]%\end{equation}
The tail probability generating function for $L_\infty$ is the Gaussian hypergeometric function
%\begin{equation}
%\label{linf}
\[
\sum_{k=0}^\infty \BP (  L_\infty > k ) z^k = \sum_{k=0}^\infty \frac{(1)_k}{(1+ \theta)_k} z^k = 
\setlength\arraycolsep{1pt}
{}_2 F_1\left(\begin{matrix}1,\,1\\ 1+\theta\end{matrix}; z\right)
\]%\end{equation}
from which the limiting mean is found to be
\[%begin{equation}
\lim_{n \to \infty} \BE ( L_n) = \BE( L_\infty) = \sum_{k=0}^\infty \BP (  L_\infty > k ) = \frac{ \theta }{(\theta - 1 )_+}
\]%end{equation}
where the last expression should be read as $\theta/(\theta - 1) < \infty$ for $\theta >1$, and $\theta/0 = \infty$ for $\theta \le 1$.
Similarly, the limiting second moment is finite only if $\theta > 2$ with the simple limit formula for the second binomial moment
\[%begin{equation}
\lim_{n \to \infty} \BE {L_n \choose 2 } = \sum_{k=0}^\infty k \, \BP (  L_\infty > k ) = \frac{ \theta }{(\theta - 1 )_+ (\theta - 2 )_+}\,.
\]%end{equation}
It appears that this pattern continues, with finite third binomial moment $2!\theta/(\theta - 3)_3$ for $\theta > 3$, and so on.

Since the number of ties at the maximum $L_n$ can be defined on a common probability space, one can also be interested
in some stronger types of convergence for these random variables than the convergence in distribution. 
The answer to this question for independent random variables is well known: Brands et~al.~\cite{MR1294106}
conjectured and Baryshnikov et al.~\cite{MR1340152} confirmed that the number $L_n$ of maxima in a sample of $n$ independent discrete
random variables can exhibit just three types of behavior as $n\to\infty$: either it converges to 1 or 
to $\infty$  in probability, or it does not have a limit. These three cases can be distinguished in terms of the 
discrete hazards \eqref{stickbreak} of the distribution of $X_1$, which are nonrandom in this case, say, $H_j=h_j$. If $h_j\to0$ as $j\to\infty$, then
the number of maxima converges in probability to 1, and this is the only possibility for convergence to a proper distribution. 
This result was extended to an almost sure convergence by Qi~\cite{MR1458007}, who showed that a.s.\ convergence holds if and only if
the series $\sum_j h_j^2$ converges. Later, a probabilistic proof of this result was given by Eisenberg \cite{MR2493010}
along with some extensions.  His results are also formulated in terms of the discrete hazards:

\begin{lemma}[\cite{MR2493010}]\label{lem:limsup} 
Let $X_1,X_2,\dots$ be a sequence of i.i.d.\ random variables with values in $\BN$ and infinitely supported distribution.
Then, for any $\ell\in\BN$, $\BP[\limsup\nolimits_n L_n=\ell]=1$ if and only if\/ $\sum_{j=1}^\infty h_j^\ell=\infty$ and\/ $\sum_{j=1}^\infty h_j^{\ell+1}<\infty$. If the above series diverge for all $\ell\in\BN$ then $\BP[\limsup\nolimits_n L_n=\infty]=1$.
\end{lemma}

This result can be immediately translated to the exchangeable $\GEM$ case, because in this case hazards are independent random variables.
Heuristically, the next Theorem means that for $\alpha\in(0,1)$, as $k$ becomes large, the $\GEM(\alpha,\theta)$ probabilities $P_k$  a.s.\ 
tend to zero regularly and the maximum is unlikely to be hit twice before a new maximal value is reached, while for $\alpha=0$ the situation is
opposite, there exist $k$ such that $P_k$ is arbitrary large compared to the tail $1-F_k$ and such values $k$ 
are repeated as maxima of the sample many times.

\begin{theorem}
Let $X_1,X_2,\dots$ have the $\GEM(\alpha,\theta)$ exchangeable distribution. 
Then 
\begin{alignat*}{2}&\BP\bigl[\limsup\nolimits_n L_n=1\bigr]=1,\qquad &&\alpha>0;\\
&\BP\bigl[\limsup\nolimits_n L_n=\infty\bigr]=1,\qquad&&\alpha=0. 
\end{alignat*}
\end{theorem}

\begin{proof}
If the distribution of $H_i$ is defined by \eqref{eq:Y} then
%\[
%\BE[H_i^k]=\frac{B(1-\alpha+k,\theta+i\alpha)}{B(1-\alpha,\theta+i\alpha)}
%=\frac{\Gamma(1-\alpha+k)\Gamma(1+(i-1)\alpha+\theta)}{\Gamma(1-\alpha)\Gamma(1+(i-1)\alpha+\theta+k)}\,.
%\]
\[
\BE[H_i^k]= \frac{( 1 - \alpha )_k} {( 1 + (i-1) \alpha + \theta )_k }\,.
%frac{B(1-\alpha+k,\theta+i\alpha)}{B(1-\alpha,\theta+i\alpha)}
%=\frac{\Gamma(1-\alpha+k)\Gamma(1+(i-1)\alpha+\theta)}{\Gamma(1-\alpha)\Gamma(1+(i-1)\alpha+\theta+k)}\,.
\]
Hence for $\alpha>0$ 
\[
%\BE[H_i^k]\sim\frac{\Gamma(1-\alpha+k)}{\Gamma(1-\alpha)}(i\alpha)^{-k},\qquad i\to\infty,
\BE[H_i^k]\sim\frac{ (1 -\alpha)_k}{ ( i \alpha)^{k} },\qquad i\to\infty,
\]
and since $H_i\in(0,1)$ by Kolmogorov's three series theorem the series $\sum H_i^2$ converges a.s. So $\BP[\limsup\nolimits_n L_n=1|(H_i)]=1$ 
by Lemma~\ref{lem:limsup}, and 
also unconditionally.  On the other hand $\BE[H_i^k]$ does not depend on $i$ for $\alpha=0$, so the series $H_i^k$ diverges by the
same theorem and again Lemma~\ref{lem:limsup} implies $\BP[\limsup\nolimits_n L_n=\infty|(H_i)]=1$ and hence unconditionally.
\end{proof}

\paragraph{Acknowledgments.}
Thanks to Daniel Dufresne for references to the infinite divisibility of $\log$ beta  distributions.

%%%%%%%%%%%%%%%%%%%%%%%%%%%%%%%%%%%%%%%%%%%%%%%%%%%%%%%%%%%%%%%%%%%
%%                                                               %%
%% Use the two commands below for producing your bibliography    %%
%% with bibtex, then comment again the commands and include the  %%
%% content of the .bbl file in this file below the commands.     %%
%%                                                               %%
%%%%%%%%%%%%%%%%%%%%%%%%%%%%%%%%%%%%%%%%%%%%%%%%%%%%%%%%%%%%%%%%%%%

\bibliographystyle{amsplain}
%%\bibliography{yourbibfilename}
%\bibliography{gemmax,0gibbs}
\bibliography{gemmax}

% add below the content of your .bbl file produced by bibtex.

\providecommand{\bysame}{\leavevmode\hbox to3em{\hrulefill}\thinspace}
\providecommand{\MR}{\relax\ifhmode\unskip\space\fi MR }
\providecommand{\MRhref}[2]{%
  \href{http://www.ams.org/mathscinet-getitem?mr=#1}{#2}
}
\providecommand{\href}[2]{#2}

%\begin{thebibliography}{10}
%\end{thebibliography}

%%%%%%%%%%%%%%%%%%%%%%%%%%%%%%%%%%%%%%%%%%%%%%%%%%%%%%%%%%%%%%%%%%%
%%                                                               %%
%% You may add acknowledgments (optional).                       %%
%%                                                               %%
%%%%%%%%%%%%%%%%%%%%%%%%%%%%%%%%%%%%%%%%%%%%%%%%%%%%%%%%%%%%%%%%%%%

%\ACKNO{We are grateful to Martin Hairer who provided a nice \texttt{MR} macro and to S\'ebastien Gou\"ezel for his useful comments on the internals of the class file.}

%%%%%%%%%%%%%%%%%%%%%%%%%%%%%%%%%%%%%%%%%%%%%%%%%%%%%%%%%%%%%%%%%%%
%%                                                               %%
%% You have reached the end of your document.                    %%
%%                                                               %%
%%%%%%%%%%%%%%%%%%%%%%%%%%%%%%%%%%%%%%%%%%%%%%%%%%%%%%%%%%%%%%%%%%%

\end{document}